\documentclass{article}

\usepackage{amsmath,amsthm,amsfonts,graphicx}
\usepackage{multirow}
\usepackage{slashbox}
\usepackage{multirow}
\def\smallddots{\mathinner{\raise7pt\hbox{.}\raise4pt\hbox{.}\raise1pt\hbox{.}}}
\def\smallsdots{\mathinner{\raise1pt\hbox{.}\raise4pt\hbox{.}\raise7pt\hbox{.}}}

\DeclareMathOperator{\diag}{diag}

\DeclareMathOperator{\rank}{rank}
\DeclareMathOperator{\nrank}{nrank}

\newtheorem{theorem}{Theorem}[section]
\numberwithin{equation}{section}
\numberwithin{table}{section}
\newtheorem{lemma}{Lemma}[section]

\newtheorem{corollary}{Corollary}[section]

\newtheorem{algorithm}{Algorithm}[section]
\newtheorem{example}{Example}[section]
\newtheorem{definition}{Definition}[section]

\newtheorem{remark}{Remark}[section]

\newtheorem{problem}{Problem}[section]
\begin{document}

\title{Fast Low-rank Approximation of a Matrix:  \\
Novel Insights, Novel Multipliers, and Extensions
%\footnote{Supported by NSF Grant CCF--1116736 and
%PSC CUNY Award  68862--00 46}
\thanks {Some results of this paper have been  
presented at 
%CASC'2015, in Aahen, Germany, September 2015, and
the Eleventh International Computer Science Symposium in Russia 
(CSR'2016),
in St. Petersbourg, Russia, June 2016}} 
\author{Victor Y. Pan} 

\author{Victor Y. Pan$^{[1, 2],[a]}$, Liang Zhao$^{[2],[b]}$, and
John Svadlenka$^{[2],[c]}$
%, and Qi Luan$^{[2],[d]}$ 
\\
\and\\
$^{[1]}$ Department of Mathematics and Computer Science \\
Lehman College of the City University of New York \\
Bronx, NY 10468 USA \\
$^{[2]}$ Ph.D. Programs in Mathematics  and Computer Science \\
The Graduate Center of the City University of New York \\
New York, NY 10036 USA \\
$^{[a]}$ victor.pan@lehman.cuny.edu \\
http://comet.lehman.cuny.edu/vpan/  \\
$^{[b]}$  lzhao1@gc.cuny.edu \\
$^{[c]}$ jsvadlenka@gradcenter.cuny.edu 
% \\ $^{[d]}$ qi-luan@yahoo.com
} 
\date{}

\maketitle

%------------------------------------------------------------------------------
%------------------------------------------------------------------------------

\begin{abstract} 

\begin{itemize}
  \item%1
%The celebrated algorithms for l
Low-rank approximation of a matrix by means of random sampling has been consistently efficient in its empirical studies by many scientists who applied it 
with various sparse and structured multipliers, but adequate formal support for this empirical phenomenon has been missing  so far.
\item%2
Our novel insight into the subject leads to such an elusive formal support and
 promises significant acceleration of the known algorithms
for some fundamental problems of matrix computations
and data mining and analysis. 
%for low-rank approximation of a matrix 
%by means of sampling more efficient multipliers.
\item%3
Our formal results and our numerical tests
 are in good accordance with each other.
\item%4
We also outline extensions of low-rank approximation
algorithms and of our progress to the Least Squares Regression, the Fast Multipole Method,
and the Conjugate Gradient algorithms.
\end{itemize}
\end{abstract}



% - - - - - - - - - - - - - - - - - - - - - - - - - - - - - - - - - - - - -

\section{Introduction}\label{sintr}

% - - - - - - - - - - - - - - - - - - - - - - - - - - - - - - - - - - - - -

\subsection{The problem of low-rank approximation and our progress briefly}\label{sprbpr}

%------------------------------------------------------------------------------

{\em Low-rank approximation of a 
matrix by means of random sampling}
is an increasingly popular subject area 
with applications to
the most
fundamental matrix computations  \cite{HMT11} as well as
 numerous problems of data mining and analysis,
``ranging from term document data to DNA SNP data" \cite{M11}.
See  \cite{HMT11}, \cite{M11}, and
\cite[Section 10.4.5]{GL13},
for surveys and ample bibliography.
% see
%\cite{GZT97}, \cite{GTZ97}, \cite{T00},
%\cite{FKV}, and \cite{DKM06},
%for sample early works.

%------------------------------------------------------------------------------

All these studies rely on the proven efficiency of random sampling with Gaussian multipliers
and mostly empirical evidence that the algorithms work as efficiently with various random sparse and structured multipliers, although  
adequate formal support for this empirical evidence has been missing so far. 

Our new insight enables such an elusive 
formal support
as well as the acceleration of low-rank approximation and some other  fundamental
 matrix computations.
In this section we outline our main results by using the  definitions
below and in the Appendix. 

% - - - - - - - - - - - - - - - - - - - - - - - - - - - - - - - - - - - - -

\subsection{Some definitions}\label{ssdef}

%------------------------------------------------------------------------------

\begin{itemize}
  \item%1 
Typically we use the concepts ``large", ``small", ``near", ``close", ``approximate", 
``ill-conditioned" and ``well-conditioned"  
quantified in the context, but we specify them quantitatively if needed. 
\item%2
Hereafter 
``$\ll$" means ``much less than" and
{\em ``flop"} stands for ``floating point arithmetic operation".
\item%3
$I_s$ is the $s\times s$ identity matrix.  $O_{k,l}$ is a $k\times l$
matrix filled with zeros. ${\bf o}$ is a vector filled with zeros.
\item%4
 $(B_1~|~\dots~|~B_h)$ 
denotes a $1\times h$ block matrix with the blocks $B_1,\dots,B_h$.
\item%5
$\diag(B_1,\dots,B_h)$
denotes a $h\times h$ block diagonal matrix with  
diagonal  blocks $B_1,\dots,B_h$.
\item%6
 $\rank(M)$, $\nrank(M)$, and $||M||$ denote the  {\em rank}, 
 {\em numerical rank}, and the {\em spectral norm} of a matrix $M$, respectively. 
[$\nrank (M)=r$ if and only if a matrix $M-E$ is well-conditioned and 
has rank $r$, for a perturbation matrix $E$
of a small norm. A matrix is ill-conditioned if and only if its rank exceeds 
its numerical rank or equivalently if and only if its small-norm perturbation
can decrease its rank.]
\item%7
$M^T$, $M^H$, and $\mathcal R(M)$ denote its   
 transpose, Hermitian  transpose, and range (column span), respectively.
\item%8
An $m\times n$ matrix $M$ is called  {\em unitary}  
if $M^HM=I_n$ or if  $MM^H=I_m$. If this matrix is known to be real, then it is also
and preferably  called {\em orthogonal}. 
 \item%9
{\em ``Likely"}  
means ``with a probability close to 1", 
the acronym ``i.i.d." stands
for
``independent identically distributed",
and we
refer to standard Gaussian random variables 
  just as {\em Gaussian}. 
\item%10
We call an $m\times n$ matrix {\em  Gaussian} and denote it $G_{m,n}$ 
if all its entries are {\em i.i.d.} Gaussian variables. 
\item%11
$\mathcal G^{m\times n}$, $\mathbb R^{m\times n}$, and $\mathbb C^{m\times n}$
 denote the classes of $m\times n$ Gaussian, real and complex matrices, respectively.
\item%12
$\mathcal G_{m,n,r}$, $\mathbb R_{m,n,r}$, and $\mathbb C_{m,n,r}$,
for $1\le r\le \min \{m,n\}$,
 denote the classes of $m\times n$ matrices $UV$
of rank at most $r$
where the matrices $U$ of size $m\times r$ and 
$V$  of size $r\times n$ are Gaussian, real and complex, respectively.
\item%13
If  $M=UV\in  \mathcal G_{m,n,r}$, then
the matrices $U$, $V$ and $M$  have rank $r$ 
with probability 1
 by virtue of Theorem \ref{thrnd}, and such a matrix 
$M$
is said to be an
 $m\times n$ {\rm factor-Gaussian}
matrix of expected rank $r$. 
 \end{itemize}
   
% - - - - - - - - - - - - - - - - - - - - - - - - - - - - - - - - - - - - -

\subsection{The basic algorithm}\label{sbsalg}

%------------------------------------------------------------------------------
%------------------------------------------------------------------------------

Recall that a
 matrix $M$ can be represented (respectively, approximated) 
by a product $UV$ of two matrices $U\in \mathbb C^{m\times r}$ 
and $V\in \mathbb C^{r\times n}$ if and only if $r\ge \rank(M)$ 
(respectively, $r\ge \nrank(M)$). The following celebrated 
%randomized 
algorithm computes such 
a representation or approximation:

\begin{algorithm}\label{alg1} {\rm \cite[Algorithm 4.1]{HMT11}.}
{\rm  Low-rank representation/approximation of a matrix.}

%------------------------------------------------------------------------------

\begin{description}

%------------------------------------------------------------------------------

\item[{\sc Input:}] 
An $m\times n$ matrix  $M$, a nonnegative tolerance $\tau$, and an integer  $r$
such that $0<r\ll \min\{m,n\}$.

%------------------------------------------------------------------------------

%\item[{\sc Output:}] 
%A  rank-$r$  approximation matrix $\tilde M$ 
%and the relative error $||\tilde M-M||/||M||$. 

%------------------------------------------------------------------------------

\item[{\sc Initialization:}] 
 Fix an 
%oversampling
 integer $p$ such that $0\le p\ll n-r$. Compute $l=r+p$.
Generate an $n\times l$  matrix $B$.

%------------------------------------------------------------------------------ 

\item[{\sc Computations:}]
%\item (cf. items 1.4.6 and 1.4.10): $~$

\begin{enumerate}
\item %1
Compute the  $m\times l$ matrix $MB$.
\item %2 
By orthogonalizing its columns, compute and output
the $m\times l$ 
matrix $Q=Q(MB)$ (cf. \cite[Theorem 5.2.3]{GL13}).
\item %3
Compute and output the $l\times n$  matrix $Q^TM$
(together with the matrix $Q$ it represents 
the $m\times n$
 matrix $\tilde M=QQ^TM$).
\item %4
Compute and output the spectral error norm 
$\Delta=||\tilde M-M||$.\footnote{Frievalds' test of \cite{F77}  enables
probabilistically estimation of the norm $\Delta$ 
at a cost of performing a bounded number
of multiplications of the matrix $\tilde M-M$ by random vectors.}
If $\Delta\le \tau$, output SUCCESS. Otherwise
output  FAILURE.
\end{enumerate}

%------------------------------------------------------------------------------

\end{description}

%------------------------------------------------------------------------------

\end{algorithm}

Hereafter we focus on Stage 1,
involving $(2n-1)ml$ flops for 
generic matrices $M$ and $B$. 
The range
of its output matrix $MB$ approximates the left  leading
singular space $\mathcal S_r$  of the input matrix $M$ associated with its 
$r$  largest singular values, and
\cite{HMT11} call this stage  ``Range finder".
Within the same asymptotic cost bound,
 \cite[Section 5]{HMT11} extends it to the approximation
of the SVD of the matrix $M$,
producing its low-rank approximation, and 
to some  other fundamental
factorizations of $M$.

It remains to specify the choice of a multiplier $B$
in order to complete the description of the algorithm.
%and this is the main challenge in our paper.
 
% - - - - - - - - - - - - - - - - - - - - - - - - - - - - - - - - - - - - -

\subsection{The choice of multipliers: basic observations}\label{smlbsc}

%------------------------------------------------------------------------------
%------------------------------------------------------------------------------

We readily prove the following  result (see Section \ref{sprth1}):

\begin{theorem}\label{thall}
Given an $m\times n$ matrix $M$ with $\nrank(M)=r$
 and a reasonably small positive tolerance  
$\tau$,  
Algorithm \ref{alg1} outputs SUCCESS
if and only if $\nrank(MB)
%=\nrank(T^T_rB)
=r$. 
%where $T^T_r$ is the matrix formed by $r$ 
%singular vectors of the matrix $M$
%associated with its $r$ largest singular values
%(hereafter we call these singular vectors {\em leading}),
%that is, if and only if the $r\times n$
%matrix $T^T_r$
%is orthogonal or nearly orthogonal to more than
%$l-r$ columns of a multiplier $B$.
\end{theorem}

Hence, for a given matrix $M$ and a proper choice of an $n\times l$ multiplier $B$,
 the algorithm solves the  {\em fixed rank problem} 
where the integer $r=\nrank(M)$ is  given
(we could have estimated this integer  by means of
binary search based on recursive application of the algorithm).

\begin{definition}\label{defbd}
For two integers $l$ and $n$, $0<l\le n$,
and any fixed 
 $n\times l$ multiplier $B$,
partition the set
 of  $m\times n$ matrices $M$
with $\nrank (M)=r$
into the sets $\mathcal M_B=\mathcal M_{B,{\rm good}}$ and
$\mathcal M_{B,{\rm bad}}$
of ``$B$-good" and ``$B$-bad" 
matrices such that $\nrank (MB)=r$
and $\nrank (MB)<r$, respectively.
\end{definition}

The following simple observations should be instructive. 

\begin{theorem}\label{thbd} (Cf. Remark \ref{rewcd}.)
Consider a vector ${\bf v}$ of dimension $n$ and 
 unitary matrices $Q$ of size
$n\times n$, $Q'$ of size $l\times l$,  and
 $B$  of size $n\times l$, so that both matrices 
$QB$ and $BQ'$ have size $l\times l$  and
 are unitary. Then 

(i) $\mathcal M_{B'}=\mathcal M_B$,

(ii) $\mathcal M_{B''}=(\mathcal M_B)Q$,

(iii) $\mathcal M_{B}\subseteq \mathcal M_{(B~|~{\bf v})}$, and

(iv) $\mathcal M_{B}$ is the set of all $m\times n$ matrices 
having numerical rank $r$ (that  is, the set $\mathcal M_{B,{\rm bad}}$
of $B$-bad matrices is
empty) if and only if $l\ge n$.
\end{theorem}

The theorem implies that  
the class $\mathcal M_B$ of $B$-good matrices 
is invariant in the map $B\rightarrow BQ'$ (by virtue of part (i)),
is invariant 
(up to its orthogonal transformation)  
 in the map $B\rightarrow QB$
(by virtue of part (ii)), 
can only grow if a column vector is appended to 
a multiplier $B$ (by virtue of part (iii)), and 
fills the whole space $\mathbb C_{m,n,r}$
or  $\mathbb R_{m,n,r}$ if $l=n$
 (by virtue of part (iv)).

%------------------------------------------------------------------------------

\subsection{\bf A recursive algorithm}\label{strrs}

%------------------------------------------------------------------------------

Theorem \ref{thbd} motivates the design of the following algorithm.

\begin{algorithm}\label{alg11} 
{\rm  Recursive low-rank representation/approximation of a matrix.}

%------------------------------------------------------------------------------

\begin{description}

%------------------------------------------------------------------------------

\item[{\sc Input:}] 
An $m\times n$ matrix  $M$ and a nonnegative tolerance $\tau$.

%------------------------------------------------------------------------------

%\item[{\sc Output:}] 
%A  rank-$r$  approximation matrix $\tilde M$ 
%and the relative error $||\tilde M-M||/||M||$. 

%------------------------------------------------------------------------------

\item[{\sc Computations:}]
%\item (cf. items 1.4.6 and 1.4.10): $~$

\begin{enumerate}
\item %1
Generate an $n\times n$  unitary matrix $B$.
\item %2 
Fix positive integers $l_1,\dots,l_h$
such that $l_1+\cdots+l_h=n$
and represent the matrix $B$ as 
a block vector $B=(B_1~|~\dots~|~B_h)$
where the block $B_i$ has size $n\times l_i$ for $i=1,\dots,h$.
\item %3
Recursively apply Algorithm \ref{alg1}
to the matrix $M$ by
substituting $l$ for $l^{(i)}=\sum_{j=1}^il_j$ 
and $B$ for $B^{(i)}=(B_1~|~\dots~|~B_i)$,
$i=1,\dots,h$. Stop when 
the algorithm outputs SUCCESS.
\end{enumerate}

%------------------------------------------------------------------------------

\end{description}

%------------------------------------------------------------------------------

\end{algorithm}

Correctness of the algorithm follows from part (iv)
of Theorem \ref{thbd}. 

%------------------------------------------------------------------------------ 

\begin{remark}\label{rewcd}
We can extend both Theorem \ref{thbd} and Algorithm \ref{alg11} 
to the case where we 
apply them to a nonsingular and well-conditioned
 (rather than unitary)
 $n\times n$
matrix $B$.
In that case all blocks $B_j$ of the multipliers $B^{(i)}$ are also well-conditioned
matrices of full rank, 
and moreover $\kappa(B_j)\le \kappa(B)$ for all $j$
(cf. \cite[Corollary 8.6.3]{GL13}).
\end{remark}  

At every recursive stage of the algorithm we can reuse the
results of the previous stages (see Section \ref{smngflr}).
If the algorithm outputs SUCCESS at Stage $l^{(h)}$, 
then its overall computational cost is bounded by that of Algorithm  \ref{alg1}
for the same input and for $l=l^{(h)}$. The bound decreases as  
 the output dimension $l^{(h)}$ decreases, which
motivates our next goal 
of the 
{\em compression of the output multipliers}, that is, of yielding
success for a smaller dimension $l^{(h)}$.

% - - - - - - - - - - - - - - - - - - - - - - - - - - - - - - - - - - - - -

\subsection{\bf Compression of output multipliers 
by means of 
Gaussian and structured randomization}\label{scmpgrs}

%------------------------------------------------------------------------------

The following theorem implies that
 Algorithm  \ref{alg11} is likely
to output SUCCESS at Stage $h$ for the smallest $h$ such that
$l^{(h)}\ge r$ if $B$ is a Gaussian matrix.

%------------------------------------------------------------------------------ 

\begin{theorem}\label{th0}  
% and \cite[Theorem 10.6]{HMT11}. 
Let  Algorithm \ref{alg1} be applied with a Gaussian multiplier $B=G_{n,l}$. Then 

(i)  $\tilde M=M$ with probability 1
%outputs a rank-$r$ representation of the matrix $M$ 
%(see Section \ref{slrrp})
if $l=r=\rank (M)$ 
(cf. Theorem \ref{thlrnk})
and 

(ii) it is likely that $\tilde M\approx M$
if $\nrank (M)=r\le l-4$ (cf. Theorem \ref{th1}).
\end{theorem}

%------------------------------------------------------------------------------

We produce an $n\times l$ Gaussian matrix $B=G_{n,l}$ by generating its $nl$ random entries,
and we pre-multiply it by a dense $m\times n$ matrix $M$ by using $ml(2n-1)$ flops.
(They dominate $O(ml^2)$ flops for the orthogonalization at Stage 2
if $l\ll n$.) 
We produce an $n\times l$ matrix $B$ of {\em subsample random Fourier or 
Hadamard transform}\footnote{Hereafter we use the acronyms {\em SRFT}  and {\em SRHT}.}  
by generating $n+l$ parameters, which  define this matrix, and we   
pre-multiply it by an $m\times n$ matrix $M$ by using $O(mn\log(l))$ flops, 
for $l$ of order $r\log (r)$
(see \cite[Sections 4.6 and 11]{HMT11},
 \cite[Section 3.1]{M11}, and  \cite{T11}).

SRFT and SRHT multipliers $B$ are {\em universal}, like Gaussian ones: 
 Algorithm \ref{alg1} applied with such a multiplier is likely to 
approximate 
closely 
a matrix $M$ having numerical rank at most $r$, 
although the transition from  Gaussian to 
SRFT and SRHT  multipliers  increases
 the estimated failure probability
 from $3\exp(-p)$, for $p\ge 4$,  to $O(1/l)$  (cf. 
\cite[Theorems 10.9 and 11.1]{HMT11}, \cite[Section 5.3.2]{M11}, and
 \cite{T11}). 

Empirically  Algorithm   \ref{alg1} 
with SRFT  multipliers 
fails very rarely even for 
$l=r+20$, 
although for some special input matrices $M$
it is likely to fail 
if $l=o(r\log (r))$ 
(cf. 
 \cite[Remark 11.2]
{HMT11} or  \cite[Section 5.3.2]
{M11}). 
%------------------------------------------------------------------------------
Researchers have consistently observed
similar empirical behavior 
of the algorithm  
 applied with SRHT and various other
 multipliers
(see \cite{HMT11}, \cite{M11}, \cite{W14}, \cite{PQY15},
and the references therein),\footnote{In view of part (ii) of Theorem \ref{thbd},
the results cited  for the classes of SRFT and SRHT
matrices also hold for the products of these classes with any unitary matrix,
in particular for the class of random $n\times l$ blocks of $n\times n$ circulant matrices
(see their definitions, e.g., in \cite{P01}).} but
so far no adequate formal support for that  empirical observation
has 
appeared in the huge bibliography 
on this highly popular subject.

% - - - - - - - - - - - - - - - - - - - - - - - - - - - - - - - - - - - - -

\subsection{\bf Our goals, our dual theorem, and its implications}\label{sglsdl}

%------------------------------------------------------------------------------

In this paper we are going to

(i) fill the void in the bibliography by supplying a missing
{\em formal support} for the cited observation, 
with far reaching implications (see parts (ii) and (iv) below), 

 (ii) define  {\em new more efficient policies of generation and application of
%sparse and structured 
multipliers}  for low-rank approximation,
% based on such formal results,   

(iii) {\em test our policies numerically}, and 

(iv) {\em extend our progress} 
to other important areas of matrix  computations.

% - - - - - - - - - - - - - - - - - - - - - - - - - - - - - - - - - - - - -
\medskip

Our basic {\em dual} Theorem \ref{th0d} below  
%Algorithm \ref{alg1}
%succeeds for average input matrix $M$ and thus in a sense for most
%of inputs {\em if a 
%multiplier  $B$ is well-conditioned and has full rank}. This {\em dual} theorem 
reverses the assumptions of our basic {\em primal} Theorem \ref{th0}
that a multiplier $B$ is Gaussian, 
while a matrix $M$ 
is fixed. 

\begin{theorem}\label{th0d} 
Let $M-E\in \mathcal G_{m,n,r}$ and 
 $||E||_2\approx 0$ (so that 
$\nrank (M)\le r$ and 
that $\nrank (M)=r$) with probability 1. Furthermore let
$B\in \mathbb R^{n\times l}$  and  
$\nrank (B)=l$. Then \\
(i)  Algorithm \ref{alg1}
outputs a rank-$r$ representation of a matrix $M$ 
with probability 1
%(see Section \ref{slrrp})
if $E=0$ and if $l=r$ and \\
(ii) it 
%is likely to
outputs a rank-$l$ approximation of that matrix 
with a probability close to 1 if $p=l-r>3$
and approaching 1 fast as the integer $p$ grows from 4.
\end{theorem}

% - - - - - - - - - - - - - - - - - - - - - - - - - - - - - - - - - - - - -

%\subsection{\bf Enhancing the efficiency of Algorithm \ref{alg1} 
%based on the duality techniques}\label{simpl}

%------------------------------------------------------------------------------

Theorem \ref{th0d} implies that 
{\em Algorithm \ref{alg1} outputs 
a low-rank approximation to the average input 
matrix $M$ that has a small numerical rank} $r$
(and thus in a sense to most of such matrices)
if the multiplier $B$ 
is a well-conditioned matrix
of full rank and if  the
 average matrix is defined under the 
Gaussian probability distribution.
The former provision, that $\nrank(B)=l$, is  natural
for otherwise we could have replaced the multiplier $B$
by an $n\times l_-$ matrix for $l_-<l$. 
 The latter  customary provision is natural
in view of the Central Limit theorem.

As an immediate implication of the theorem, 
on the average input $M$
Algorithm \ref{alg11} applied with any unitary 
or just nonsingular and
well-conditioned $n\times n$ multiplier $B$
also outputs SUCCESS 
at its  earliest recursive Stage $h$
at which the dimension $l^{(h)}=\sum_{j=1}^h l_j$ exceeds $r-1$.
Such application
 can be viewed as {\em derandomization} of 
Algorithms \ref{alg1} and \ref{alg11} 
versus their common application with  Gaussian sampling.
%Furthermore the set of bad inputs $\mathbb C^{m\times n}-M_B$
%is compressed fast as the output dimension $l_i^+$ is increasing above $r+3$.

In Sections \ref{smngflr} and \ref{ssprsml} we specify some promising policies
of the generation of multipliers based on these observations
and in Section \ref{ststs} perform numerical tests for 
such policies.

%------------------------------------------------------------------------------

\subsection{\bf Related work, our novelties, and extension of our progress}\label{sextprg}

%------------------------------------------------------------------------------

Part (ii) of our Theorem   \ref{th0} is
implied by \cite[Theorem 10.8]{HMT11}, but our specific supporting estimates are 
 more compact, cover the case of any $l\ge r$
(whereas \cite{HMT11} assumes that $l\ge r+4$),
and we deduce them by using a shorter proof (see Remark \ref{repfr}).    
Our approach, our results listed in Section \ref{sglsdl}, 
some of our  techniques, e.g., expan\-sion/compres\-sion 
in Section \ref{smngflr},
and even 
the concept of factor-Gaussian matrices 
are new, and so are our
families of multipliers and 
policies of their generation, combination, and application 
in Sections \ref{smngflr} and \ref{ssprsml} as well.

Moreover, our progress  
can be extended to a  variety of 
important matrix computations.
In Section \ref{sext} (Conclusions) 
we outline such  novel extensions
 to
%the acceleration of
 the highly popular and much studied
computations for Least Squares Regression,  the 
 Fast Multipole Method 
%of \cite{GR87}, \cite{CGR88},
% (listed as one of the 10 most important algorithms of 
%the 20th century),
and the Conjugate Gradient Algorithms.\footnote{Hereafter we use the acronyms 
``LSR" for ``Least Squares Regression", ``FMM"
for ``Fast Multipole Method", 
and ``{\em CG}" for ``Conjugate Gradient".}
%In our paper \cite{PZa} we advance
%Gaussian elimination with no pivoting as well as block Gaussian elimination
%by applying techniques similar to the
%ones of this paper.
The extensions
 provide new insights
and new opportunities and should motivate 
further effort and further progress.

% - - - - - - - - - - - - - - - - - - - - - - - - - - - - - - - - - - - - -

\subsection{\bf Organization of the paper}\label{sorg}

%------------------------------------------------------------------------------

We organize our presentation as follows: 

\begin{itemize}
\item%1
In Section  \ref{smngflr} we describe some policies for
managing rare failures of Algorithm \ref{alg1}.
\item%2
In Section \ref{ssprsml} we present some efficient multipliers 
for low-rank approximation.
\item%3
In Section \ref{sprf} we
prove Theorems \ref{thall}, \ref{th0}, and \ref{th0d}. 
\item%4
Section \ref{ststs}
(the contribution of the second and the third authors)
covers our numerical tests.
\item%5
In Section \ref{sext} we extend our approach
to the acceleration of the LSR, FMM,
and CG algorithms.
\item%6
The Appendix covers basic definitions, auxiliary results, and
low-rank representation of a matrix. 
%by means of random sampling.

\end{itemize}

%------------------------------------------------------------------------------

\section{Preventing and managing unlikely failure of Algorithm \ref{alg1}}\label{smngflr} 

%------------------------------------------------------------------------------
%------------------------------------------------------------------------------

%\subsection{Possible failure of Algorithm \ref{alg1}
%in case of sparse inputs and multipliers} \label{schfl}

%-----------------------------------------------------------------------------

{\bf Two conflicting goals and a simple recipe.}

We try to decrease:

(i) the cost of generation of a multiplier $B$ and of the computation of the 
matrices $MB$, $Q=Q(MB)$, $Q^TM$, and $QQ^TM-M$ and

(ii)  the chances for failure of Algorithm \ref{alg1}.

Towards these goals we should seek 
sparse and structured multipliers $B_1,\dots,B_h$ 
such that  Algorithm \ref{alg11} would succeed
already for $h=1$.
% or at least for a small
%integer positive $h$.
 Theorem \ref{th0d} implies that  
this holds for the average $m\times n$ matrix $M$  
having a small numerical rank $r$
whenever we apply an  $n\times l$ well-conditioned  multiplier $B$
of full rank $l\ge r$. Thus we readily achieve our goals
except for relatively rare  special matrices $M$. 

Real computations can deal with  ``rare special" 
input matrices $M$, not covered by Theorem \ref{th0d},
but in our extensive tests in Section \ref{ststs}
with a variety
of inputs, our small collection of sparse and structured multipliers in the next section
turned out to be powerful enough for fulfilling our goals. 

In the rest of this section we discuss the choice of the size 
of the multipliers and the policies of handling the cases where 
Algorithm \ref{alg11} needs more than a single recursive step in order to succeed.

\medskip

{\bf The size of multipliers.}
In the case of
generic input matrix $M$,
Algorithm \ref{alg1} has the same chances for success 
with any well-conditioned multiplier $B$ of full rank $l$,
by virtue of Theorem \ref{thbd}. 
Empirically these chances  are quite high and increase fast as
the integer parameter $p=l-r$ grows,
but then
the cost of computing the matrices
$MB$, $Q=Q(MB)$, $Q^TM$, and $QQ^TM-M$ increases as well,
contrary to our goal (i). So heuristic choice of 
a positive integer $p$ below 20 is preferred:
 ``setting $l=r+10$ or $l=r+20$
is typically more than  adequate"
according to \cite[Section 4.6]{HMT11} and to
our extensive tests.
 
\medskip

{\bf Reusing multipliers.}
Recall from 
\cite[Sections 8 and 9]{HMT11} that for all $j$
the matrices
$\tilde M^{(j)}=Q^{(j)}Q^{(j)T}M$
 and $\tilde M_j=Q_jQ_j^TM$, for $Q^{(j)}=Q(MB^{(j)})$,
 $Q_j=Q(MB_j)$, and $B^{(j)}=(B_1~|~\dots~|~B_j)$,
are the orthogonal projections 
of the matrices 
$MB^{(j)}$ and $MB_j$, 
respectively, onto the range of the matrix $M$.
Hence  
 $M-\tilde M^{(h)}=M-\tilde M^{(h-1)}-\tilde M_h$,
and so  at the $h$th stage of Algorithm \ref{alg11}, for $h>1$,
we can reuse such projections computed at its stage $h-1$
rather than recompute them.

\medskip
 
{\bf Compression of low-rank approximations via expansion and with heuristic amendment.}

Consider the following
{\bf expansion/compression algorithm:}
\begin{enumerate}
\item%1
Fix a larger dimension $l_+$, generate a sparse and structured   
$n\times l_+$ multiplier $B$, and compute the $m\times l_+$ product $MB$.
\item%2
Fix an integer $l$, $r\le l\le f l_+$ for  a reasonably small fraction $f$, 
generate a Gaussian $l_+\times l$
multiplier $G$ and compress the product $MB$ into 
the $m\times l$ matrix $MBG$. 
\end{enumerate}
For appropriate sparse and structured multipliers $B$, Stage 1 is inexpensive
even for $l_+=n$. Moreover, by virtue of Theorem \ref{th0d}, 
the class of input 
matrices $M$ for which $\mathcal (MB)\approx \mathcal S_r$  
is large and 
 grows fast as 
the integer  $l_+-r$ increases from 4. 
Stage 2 produces a matrix $MBG$, 
such  that it is likely that $\mathcal (MBG)\approx \mathcal S_r$
by virtue of Theorem \ref{th0} and
involves $(2l_+-1)ml$  flops, which is relatively
inexpensive if $l-r$ is not large and  if $l_+\ll \min\{m,n\}$.
Furthermore we can decrease this cost by applying a SRHT or STRF
multiplier $G$ instead of Gaussian (cf. \cite[Section 4.6]{HMT11}).

In our extensive tests in  Section \ref{ststs},
 even a simpler heuristic solution was always sufficient for the
compression of the output.
Namely, in our tests, Algorithm \ref{alg11} 
consistently succeeded after trying $h<4$ multipliers 
properly chosen (in line with Theorem \ref{th0d}),
and for $h>1$ we produced a desired compressed low-rank approximation
by applying the following heuristic amendment to the algorithm:

\medskip 

{\em Amendment 1 (combining failed multipliers):} 
if Algorithm \ref{alg11} has succeeded in $h>1$ recursive steps,
then re-apply it with the sum or linear combination of the $h$ multipliers
$B_1,\dots,B_h$ involved into these steps.

%\medskip 

%{\em Amendment 2 ( multipliers of variable sizes)}:
%if Algorithm \ref{alg1} has failed
% with a multiplier $B=B_0$, re-apply
%it  with multipliers $\widehat B_1=(B_0~|~B_1)$, $\widehat B_2=(B_0~|~B_1~|~B_2)$, and so on,
%until success, where the matrices $B_i$ have sizes $n\times l_i$,
%for $l_0=l>l_1>l_2> \dots>l_q\ge 1$, $i=0,1,\dots,q$, and some integer $q$. 

%------------------------------------------------------------------------------

\section{Generation of Multipliers}\label{ssprsml}

%------------------------------------------------------------------------------

\subsection{Multipliers  as rectangular submatrices}\label{srndmlt}

%------------------------------------------------------------------------------

Given 
%an input $m\times n$ matrix $M$ and
two integers $l$ and $n$,  $l\ll n$, we first generate
four classes of  very sparse primitive $n\times n$ matrices,
then combine them into some basic families of $n\times n$ matrices 
(we denote them $\widehat B$ in this section),
and finally define efficient $n\times l$ multipliers as  submatrices $B$,
  made up of $l$ fixed (e.g., leftmost) or random 
columns of the matrices $\widehat B$. In this case
%or in a combination of such matrices (see Section \ref{smngflr}).
 $\kappa(B)\le \kappa(\widehat B)$ 
(cf. \cite[Theorem 8.6.3]{GL13}). 

Random choice of $l$ columns involves $l$ random parameters, 
but in the transition $\widehat B\rightarrow B$
we lose  $n-l$ columns of the matrix  $\widehat B$
together with all random parameters (if any)
located in these columns. The arithmetic cost 
of the computation of the product $MB$
does not increase in the transition $\widehat B\rightarrow B$,
but can stay invariant, e.g., 
 if the matrix $\widehat B$ is structured and if the transition is by means
 of random choice of $l$ columns of the matrix $\widehat B$. 

The transition  $\widehat B\rightarrow B$ to a 
$p\times q$ matrix  $B$ where $\max\{p,q\}<n$
can be numerically unstable: it can increase the  condition number dramatically.
This observation restricts the proposed basic families (i) and (iii)
of $n\times n$ multipliers of Sections \ref{shad} and \ref{scrcabr}
 to the case where
$n$ is a power of 2. By using block representation 
of a matrix $\widehat B$, however, we can 
relax this restriction
 in the cases where   
$n=\sum_{i=1}^q(-1)^{m_i}2^{k_i}$ is an algebraic sum of a small number $q$ 
of powers of 2: in such cases we can use  matrices of basic families (i) and (iii)
as blocks of sizes $2^u\times 2^v$
for integers $u$ and $v$.
We refer the reader to  
 \cite{M11}, \cite{W14},  and
the bibliography therein 
for various other techniques for relaxing the restrictions
on the  size of multipliers. 

%------------------------------------------------------------------------------

\subsection{$n\times n$ matrices of four  primitive types}\label{sdfcnd}

%------------------------------------------------------------------------------
  
%Next we describe four matrix primitives of size $n\times n$. 

%------------------------------------------------------------------------------

%$\big (\begin{smallmatrix} 1& ~~1\\ 1&-1\end{smallmatrix}\big )$

\begin{enumerate}
  \item%1
Fixed  or random
 {\em permutation matrix} $P$.
\item%2
A {\em diagonal matrix}  $D=\diag(d_i)_{i=0}^{n-1}$, with 
fixed  or random
diagonal entries $d_i$ such that
$|d_i|=1$ for all $i$ (and so each of $n$ entries $d_i$ lies on the unit circle $\{x:~|z|=1\}$, 
being either nonreal or  $\pm 1$).
\item%3
An $f$-{\em circular shift matrix} 
$Z_f=\big (\begin{smallmatrix} {\bf 0}^T  & ~f\\I_{n-1}   & ~{\bf 0}\end{smallmatrix}\big )$
and its transpose $Z_f^T$ for a fixed scalar $f$ such that either $f=0$ or $|f|=1$.
We write $Z=Z_0$, call $Z$ unit {\em down-shift matrix}, and call $Z_1$ unit {\em circulant matrix}. 
\item%4
A   $2s\times 2s$ {\em Hadamard  primitive matrix}
%\begin{equation}\label(eqprimhad)
$H^{(2s)}=\big (\begin{smallmatrix} I_s  & ~I_s  \\
I_s   & -I_s\end{smallmatrix}\big )$
%\end{equation}
 for a positive integer $s$
 (cf. \cite{M11},  \cite{W14}).
\end{enumerate}

All these multipliers are very sparse, have
nonzero entries evenly distributed 
throughout them,  and can be
pre-multiplied by a vector
by using from 0 to $2n$ flops per a multiplier.
 All of them,
 except for the matrix $Z$, are unitary (or real orthogonal). Hence,
for  the average input matrix $M$, 
Algorithm \ref{alg1} succeeds with any of their $n\times l$ submatrix $B$
by virtue of  Theorem \ref{th0d},
and similarly with any $n\times l$ submatrix of $Z$ of full rank.
For sparse matrices $M$, 
the algorithm can  fail with these sparse multipliers 
in good accordance with our tests
(cf. Example \ref{ex1} below),
 but in all these tests we have consistently 
 succeeded with 
fixed (e.g., leading) or random  $n\times l$ submatrices
of various simple
{\em combinations} of our four primitives.

%------------------------------------------------------------------------------

\begin{example}\label{ex1}
Let  $B$ denote   an $n\times r$
matrix.
Then $\rank(MB)=r$ for  
$M=\diag(I_r,O_{m-r,n-r})P$ 
and all permutation matrices $P$
if and
only if all $r\times r$ submatrices
of the multiplier $B$ are nonsingular.
\end{example}

%------------------------------------------------------------------------------
%------------------------------------------------------------------------------

\subsection{Basic combinations of  primitive matrices: general observations}\label{sdfcndf}

%------------------------------------------------------------------------------

In the next subsections, by combining primitives 1--4, we
define families of $n\times n$ 
sparse and/or structured  matrices 
%$\widehat B$
whose $n\times l$ submatrices $B$
 have been successfully tested as 
multipliers in Section \ref{ststs}.
% in our tests of Algorithm \ref{alg1}.

For the choice of permutation and diagonal primitive 
matrices 
 trade-off is possible:  
by choosing the identity 
matrix $I_n$ for these primitives we decrease
the computational cost of the generation and application
of the multipliers, whereas by choosing
 random primitives
 we 
increase this cost but also increase the chances for success of  
Algorithm \ref{alg1}.
   
%------------------------------------------------------------------------------

\subsection{Family (i): 
%sparse ARSPH matrices
multipliers based on the Hadamard and Fourier processes}\label{shad}

%------------------------------------------------------------------------------

  At first we recursively define the 
dense and orthogonal (up to scaling by constants) $n\times n$ matrices $H_n$ 
of {\em Walsh-Hadamard transform} for $n=2^k$ as follows
(cf.   \cite[Section 3.1]{M11} and our Remark \ref{recmb}): 
\begin{equation}\label{eqrfd}
H_{2q}=\begin{pmatrix}
H_{q} & H_{q} \\
H_{q} & -H_{q}
  \end{pmatrix}
%H_{2^{k-d}=\begin{pmatrix}
%I_{2^{k-d} & I_{2^{k-d}  \\
%I_{2^{k-d} & -I_{2^{k-d} 
%\end{pmatrix},
\end{equation}
for $q=2^h$, $h=0,1,\dots,k-1$,
  and the Hadamard  primitive matrix
 $H_{2}=H^{(2)}=
\big (\begin{smallmatrix} 1  & ~~1  \\
1   & -1\end{smallmatrix}\big )$  of type 4
%(cf. (\ref{eqprimhad})
 for $s=1$.  

We can pre-multiply such a matrix $H_{n}$ by a vector by using $nk$ additions
and subtractions for $n=2^k$.

Next we sparsify it by defining it by a 
 shorter recursive process, that is, 
by fixing a  {\em recursion depth} $d$, $0<d<k$, and  applying equation (\ref{eqrfd}) where
$q=2^h$, $h=k-d,k-d+1,\dots,k-1$,  and $H_{2s}$ for $2s=2^{k-d+1}$
is the Hadamard  primitive matrix $H^{(2s)}$ of type 4.
For  $n=2^k$,  we denote the resulting $n\times n$ matrix
$H_{n,d}$ and  for $1\le d< k$ call it 
 $d$--{\em Abridged Hadamard
 (AH)
  matrix}. 

This matrix  
is still orthogonal (up to scaling),
 has $q=2^d$ nonzero entries 
in every row and  column, and hence
 is sparse unless 
$k-d$ is a small integer.
Pre-multiplication of such a  matrix by a vector uses  $dn$
additions
and subtractions
and allows
highly efficient  
 parallel implementation 
(cf. 
Remark \ref{resprs0}).

We similarly obtain sparse matrices
by shortening a recursive process
of the generation of the $n\times n$ matrix  $\Omega_n$
of {\em discrete Fourier transform (DFT)} at $n$ points,  for $n=2^k$:
\begin{equation}\label{eqdft}
\Omega_n=(\omega_{n}^{ij})_{i,j=0}^{n-1},~{\rm for}~n=2^k~{\rm and~a~primitive}~
n{\rm th~root~of~unity}~\omega_{n}=\exp(2\pi\sqrt {-1}/n).
\end{equation}

The matrix $\Omega_n$ is unitary up to scaling by $\frac{1}{\sqrt n}$.
We can multiply it by a vector by using $1.5nk$ flops
(by applying FFT) and can  generate it recursively as follows 
(cf. \cite[Section 2.3]{P01}  and our Remark \ref{recmb}):\footnote{Transposition of the matrices of this representation of FFT, called decimation 
in frequency (DIF) radix-2 representation, turns them into the matrices of the alternative
representation of FFT, called
decimation 
in time (DIT) radix-2 representation.}
\begin{equation}\label{eqfd}
\Omega_{2q}=
\widehat P_{2q}
\begin{pmatrix}\Omega_{q}&~~\Omega_{q}\\ 
\Omega_{q}\widehat D_{q}&-\Omega_{q}\widehat D_{q}\end{pmatrix},
\end{equation}
for $q=2^h$, $h=0,1,\dots,k$,  and
$\Omega_{2}=H_{2}=\big (\begin{smallmatrix} 1  & ~~1  \\
1   & -1\end{smallmatrix}\big ).$

Then again  we can sparsify this matrix by defining it by a 
 shorter recursive process, that is, 
by fixing a recursion depth $d$, $0<d<k$, and  applying equation (\ref{eqfd}) for
$q=2^h$, $h=k-d,k-d+1,\dots,k-1$,  and $\Omega_{2s}$ for $2s=2^{k-d+1}$
being the Hadamard  primitive matrix $H^{(2s)}$ of type 4.

For $1\le d\le k$ and $n=2^k$,   
we denote the resulting $n\times n$ matrix
$\Omega_{n,d}$ and call it 
 $d$-{\em Abridged   Fourier
 (AF)
  matrix}. It is also unitary (up to scaling),
 has $q=2^d$ nonzero entries 
in every row and column, and thus is
  sparse unless 
$k-d$ is a small integer. Pre-multiplication of such a matrix by a vector involves  $1.5dn$
flops
and agian allows
highly efficient  
 parallel implementation (cf. 
Remark \ref{resprs0}).

By applying fixed or random permutation and/or scaling to AH matrices
$H_n$ and  AF matrices $\Omega_n$, we obtain the 
sub-families of 
$d$--{\em Abridged Scaled and Permuted 
 Hadamard (ASPH)} matrices, $PDH_n$, and 
$d$--{\em Abridged Scaled and Permuted 
Fourier  (ASPF)} $n\times n$
 matrices, $PD\Omega_n$. Then  we define the families of
ASH, ASF, APH, and APF matrices, 
 $DH_n$, $D\Omega_n$, $PH_n$, and  $P\Omega_n$, respectively.
Each random permutation or scaling  
 contributes up to $n$ random parameters.  

\begin{remark}\label{recmb}
We can readily express representations (\ref{eqrfd}) and (\ref{eqfd})
as combinations of our  primitives 1--4:
$$H_{2q}=\diag(H_q,H_q)H^{(2q)}~{\rm and}~\Omega_{2q}=
\widehat P_{2q} \diag(\Omega_{q},\Omega_{q}\widehat D_q)H^{(2q)}$$
where $H^{(2q)}$ denotes a $2q\times 2q$ Hadamard's primitive matrix of type 4.
We can introduce more random parameters
by means of random recursive permutations and diagonal scaling as follows:
$$\widehat H_{2q}=P_{2q}D_{2q}\diag(\widehat H_q,\widehat H_q)H^{(2q)}~{\rm and}~
\widehat \Omega_{2q}=
P_{2q}D_{2q} \diag(\Omega_{q},\Omega_{q}\widehat D_q)H^{(2q)}$$
where $P_{2q}$ are $2q\times 2q$ random permutation matrices of primitive class 1
and $D_{2q}$ are $2q\times 2q$ random matrices of diagonal scaling of primitive class 2,
for all $q$. We need at most $2dn$ additions and subtractions
in order to pre-multiply a matrix $\widehat H_{n}$ by a vector
and at most $2.5dn$ flops in order to pre-multiply a matrix $\widehat H_{n}$ by a vector,
in both cases for $n$ being a power of 2.
\end{remark}

%------------------------------------------------------------------------------

\subsection{ $f$-circulant and sparse   
$f$-circulant matrices}\label{scrcsp}

%------------------------------------------------------------------------------

 Recall that an
 {\em $f$-circulant matrix}
$Z_f({\bf v})=\sum_{i=0}^{n-1}v_iZ_f^i$, 
for  
the matrix $Z_f$ of $f$-circular shift,
is defined by a 
scalar $f\neq 0$  and by
the first column ${\bf v}=(v_i)_{i=0}^{n-1}$ and
 is called {\em  circulant} if $f=1$ and {\em skew-circulant} if $f=-1$.
Such a matrix is nonsingular with probability 1 (see Theorem \ref{thrnd}) and
is likely to be well-conditioned \cite{PSZ15}
if  $|f|=1$  and if the vector ${\bf v}$ is Gaussian or is random and uniformly bounded.

{\bf FAMILY (ii)}  of  {\em sparse} $f$-{\em circulant matrices} 
$\widehat B=Z_f({\bf v})$ is
 defined by a fixed or random scalar $f$, $|f|=1$, and by
the  first column having exactly 
$q$ nonzero entries, for $q\ll n$.
The positions and values of nonzeros can be
 randomized (and then the matrix would depend on up to $2n+1$ random values).

Such a matrix can be pre-multiplied by a vector by using at most 
$(2q-1)n$ flops  or, in the real case where $f=\pm 1$ and $v_i=\pm 1$
for all $i$, by using at most
$qn$ additions and subtractions. 

The same cost estimates apply to the  generalization 
of such a matrix $Z_f({\bf v})$ to
any    
 sparse matrix with exactly $q$ nonzeros entries $\pm 1$ 
in every row and in every column for $1\le q\ll n$, which
can be defined as the sum $\sum_{i=1}^q\widehat D_iP_i$
for fixed or random matrices  $P_i$  and $\widehat D_i$ of  primitive  types 1 and 2,
respectively.

%------------------------------------------------------------------------------

\subsection{Abridged   
$f$-circulant matrices}\label{scrcabr}

%------------------------------------------------------------------------------

First recall the following well-known expression for a
 $g$-circulant matrix: 
$$Z_g({\bf v})=\sum_{i=0}^{n-1}v_iZ_g^i=D_f^{-1}\Omega_n^HD\Omega_nD_f$$
where $g=f^n$,  $D_f=\diag(f^i)_{i=0}^{n-1}$, ${\bf v}=(v_i)_{i=0}^{n-1}=(\Omega_nD_f)^{-1}{\bf d}$, 
${\bf d}=(d_i)_{i=0}^{n-1}$, and
$D=\diag(d_i)_{i=0}^{n-1}$
(cf. \cite[Theorem 2.6.4]{P01}).
For $f=1$, the expression is simplified: $g=1$, $D_f=I_n$, $F=\Omega_n$, and
$Z_g({\bf v})=\sum_{i=0}^{n-1}v_iZ_1^i$
is a circulant matrix:
$$Z_1({\bf v})=\Omega_n^HD\Omega_n,~D=\diag(d_i)_{i=0}^{n-1},~{\rm for}~ 
{\bf d}=(d_i)_{i=0}^{n-1}=\Omega_n{\bf v}.$$

Pre-multiplication of an $f$-circulant matrix by a vector 
is reduced to pre-multiplication of the matrices $\Omega$
and $\Omega^H$ by two vectors and in  addition 
to $4n$ flops (or $2n$ in case of a circulant matrix).
 This involves $O(n\log (n))$ flops overall
and then again allows highly efficient
 parallel implementation
(see Remark \ref{resprs0}).

For a fixed scalar $f$, we can define the matrix $C_g({\bf v})$ by 
 any of the vectors ${\bf v}$ or  ${\bf d}=(d_i)_{i=0}^{n-1}$. 
The matrix is unitary (up to scaling) if $|f|=1$ and if $|d_i|=1$ for all $i$.
The family of such matrices is defined by $n+1$ real parameters 
(or by $n$ such parameters for a fixed $f$)
which we can fix or choose at random.

Now suppose that $n=2^k$, $1\le d<k$, $d$ and $k$ are integers,
and substitute a pair of AF
matrices of recursion length $d$ for two factors $\Omega_n$ in the
above expressions. 
Then the resulting {\em abridged $f$-circulant matrix} $C_{g,d}({\bf v})$
{\em  of recursion depth} $d$ is still unitary  (up to scaling),
defined by $n+1$ or $n$ parameters $d_i$ and $f$, 
is sparse unless the positive integer $k-d$ is  small,
and can be pre-multiplied by a vector by using $(3d+3)n$ flops.
Instead of AF matrices, we can substitute  a pair of 
ASPF, APF, ASF, AH,
ASPH, APH, or ASF
matrices
for the factors
$\Omega_n$. 
Such matrices form {\bf FAMILY (iii)} of
 $d$--{\em abridged $f$-circulant matrices}.

%------------------------------------------------------------------------------

\subsection{Inverses of bidiagonal matrices  
%modified pairs of Householder reflections
}\label{sinvchh}

%------------------------------------------------------------------------------

{\bf FAMILY (iv)}  is formed by the {\em inverses of bidiagonal matrices}
$$\widehat B=(I_n+DZ)^{-1}~{\rm or}~(I_n+Z^TD)^{-1}$$ for 
%fixed or random
  %$Z=\big (\begin{smallmatrix}  
       %0  & ~0\\
        %I_{n-1}   & ~0 
    %\end{smallmatrix}\big )$ 
% the  $n\times n$ matrix filled with zeros,
%except for the $n-1$ entries of its first subdiagonal,
% filled with ones.
a matrix $D$  of   primitive  type 2 and the down-shift matrix
$Z$. 

We can randomize the matrix $\widehat B$ 
 by choosing up to $n-1$ random diagonal entries of
the matrix $D$
(whose leading entry  makes no impact on $\widehat B$) and 
%so $||B||\le \sqrt n$
%=({\bf r}_j)_{j=1}^l$
 can pre-multiply the matrix $\widehat B$  by
a vector by using $2n-1$ flops or, in the real case, just $n-1$ additions and subtractions.

$||\widehat B||\le \sqrt {n}$ because nonzero entries of the  lower triangular 
matrix $\widehat B=(I_n+DZ)^{-1}$ have absolute values 1,  and
clearly $||\widehat B^{-1}||=||I_n+DZ||\le \sqrt 2$. Hence 
$\kappa(\widehat B)=||\widehat B||~||\widehat B^{-1}||$
 (the spectral condition number of  $\widehat B$) cannot exceed
$\sqrt {2n}$ for $\widehat B=(I_n+DZ)^{-1}$,
and the same bound holds for $\widehat B=(I_n+Z^TD)^{-1}$.
 
%------------------------------------------------------------------------------

\subsection{Summary of estimated numbers of flops and random variables involved}\label{sflprnd}

%------------------------------------------------------------------------------

Table \ref{tabmlt}  shows 
upper bounds on (a) the numbers of random variables involved into the matrices $\widehat B$
of the four families (i)--(iv)
and (b) the numbers of flops for pre-multiplication of such a matrix by
 a vector.\footnote{The asterisks in the table
% \ref{tabmlt} 
show that the matrices 
of families (i) AF, (i) ASPF, and (iii) involve nonreal roots of unity.}
%These data include  $n$ random parameters defining a random  $n\times n$
%permutation  or diagonal matrix.
For comparison, using a  Gaussian $n\times n$  multiplier  involves  $n^2$ random variables 
and $(2n-1)n$ flops.
%\end{remark}

%------------------------------------------------------------------------------

\begin{table}[ht] 
  \caption{The numbers of random variables and flops}
\label{tabmlt}

  \begin{center}
    \begin{tabular}{|*{8}{c|}}
      \hline
family &  (i) AH & (i) ASPH & (i) AF &  (i) ASPF & (ii) &  (iii)& (iv)  
\\ \hline
random variables & 0  &$2n$ & 0  & $2n$ & $2q+1$  
& $n$ &   $n-1$  
\\\hline
flops complex  &  $dn$ & $(d+1)n$ & $1.5dn$ & $(1.5d+1)n$  & $(2q-1)n$     & $(3d+2)n$ &   $2n-1$
 
\\\hline
 flops in real case & $dn$  & $(d+1)n$  & * & *  & $qn$ &  *    
 &  $n-1$ 
\\\hline

    \end{tabular}
  \end{center}
\end{table}

\begin{remark}\label{resprs0}
Other observations besides flop estimates  can be  decisive.
E. g., a
special recursive structure 
 of  an ARSPH matrix $H_{2^{k},d}$ and 
an ARSPF matrix $\Omega_{2^{k},d}$
allows
highly efficient  
 parallel implementation of  
their pre-multiplication by a vector based on 
Application Specific Integrated Circuits (ASICs) and 
Field-Programmable Gate Arrays (FPGAs), incorporating Butterfly
Circuits \cite{DE}.
\end{remark}

\begin{remark}\label{resprs01}
We are likely to save flops for the 
 approximation of the product $M\widehat B$
if we use
leverage scores \cite{W14}
(a.k.a. sampling probabilities \cite[Sections 3 and 5]{M11}).
\end{remark}

%------------------------------------------------------------------------------

\subsection{Other basic families}\label{sbscfml}

%------------------------------------------------------------------------------

There is a number of other interesting basic matrix families.
According to \cite[Remark 4.6]{HMT11}, ``among the structured random matrices ....
one of the strongest candidates involves sequences of random Givens rotations".
The sequences are  defined in factored form,
with two of the factors being the products of $n-1$ Givens rotations each,
two being permutation  matrices, three  matrices of
diagonal scaling, and the DFT matrix $\Omega_n$.
The DFT factor makes the resulting matrices dense, but we can make them sparse
by  replacing that factor by an 
 AF, ASF, APF, or ASPF matrix of recursion depth $d<\log_2(n)$. This would also decrease 
the number of flops  involved in pre-multiplication of such a multiplier by a vector
from order $n\log_2(n)$ to $1.5dn+O(n)$. 

We can obtain  new candidate families of efficient multipliers by replacing 
either or both of the Givens products with sparse matrices of Householder reflections
 matrices  of the form
$I_n-\frac{2{\bf h}{\bf h}^T}{{\bf h}^T}$
for fixed or random 
vectors ${\bf h}$ (cf. \cite[Section 5.1]{GL13}).
We can make these matrices sparse by choosing sparse vectors ${\bf h}$.

We  can  obtain a great variety of the families of multipliers promising to be efficient 
if we properly combine the matrices of basic families (i)--(iv) extended 
by the above matrices. This can be just linear  combinations,
but we can also use block representation as in the following    
real  
$2\times 2$ block matrix $\frac{1}{\sqrt n}\begin{pmatrix}
Z_1({\bf u}) & Z_1({\bf v})  \\
Z_1({\bf v}) & -Z_1({\bf u}) 
\end{pmatrix}D$
for two  vectors ${\bf u}$ and ${\bf v}$
and a  matrix $D$ of primitive class 2.

 The reader 
can find other useful families of multipliers
in our Section \ref{ststs}. E.g., according to our tests
in  Section \ref{ststs}, it turned out to be  efficient
to use  nonsingular well-conditioned (rather than unitary)
diagonal factors in the definition of some of our basic matrix families.

%------------------------------------------------------------------------------

\section{Proof of  Theorems \ref{thall}, \ref{th0}, and \ref{th0d}}\label{sprf}

%------------------------------------------------------------------------------

\subsection{Low-rank representation: proof}\label{slrr}

%------------------------------------------------------------------------------

%In this section we prove our basic Theorems \ref{th0} and  \ref{th0d}.

\begin{theorem}\label{thlrnk} 
(i) For an $m\times n$ input matrix $M$ of rank $r\le n\le m$,
its rank-$r$ representation is given by the products
%\begin{equation}\label{eqverif}
$R(R^TR)^{-1}R^TM=Q(R)Q(R)^TM$
%\end{equation}
provided that $R$ is an $n\times r$ matrix such that  $\mathcal R(R)=\mathcal R(M)$
and that $Q(R)$ is a matrix obtained by means of column orthogonalization of $R$.  

(ii)  $\mathcal R(R)=\mathcal R(M)$,  
for $R=MB$ and an $n\times r$ matrix $B$, 
with probability $1$
if $B$ is Gaussian and 

(iii) with a probability at least $1-r/|S|$ if an $n\times r$ matrix $B$ 
has i.i.d. random entries sampled 
 uniformly from a finite set $\mathcal S$
of cardinality $|S|$.
\end{theorem}
\begin{proof}
Readily verify part (i). Then note that
$\mathcal R(MB)\subseteq\mathcal R(M)$, for an $n\times r$ 
multiplier $B$.
Hence $\mathcal R(MB)=\mathcal R(M)$ if and only if $\rank(MB)=r$,
and therefore if and only if a multiplier $B$ has full rank $r$.

Now parts (ii) and (iii) follow from
 Theorem \ref{thrnd}.
\end{proof}

Parts (i) and (ii)  of Theorem \ref{thlrnk} 
imply parts (i) of Theorems \ref{th0} and  \ref{th0d}.

%------------------------------------------------------------------------------

\subsection{From low-rank representation to  low-rank approximation: a basic step}\label{slrrlra}

%------------------------------------------------------------------------------

Our extension of the above results to the proof of  Theorems  \ref{thall},
 \ref{th0}, and  \ref{th0d}  relies on our next 
lemma and theorem.
Hereafter $\sigma_j(M)$ denotes the $j$th largest singular value of the matrix $M$
(cf. Appendix \ref{sdef}).

%------------------------------------------------------------------------------

\begin{lemma}\label{letrnc} (Cf. \cite[Theorem 2.4.8]{GL13}.)
For an integer $r$ and an $m\times n$ matrix $M$ where $m\ge n>r>0$,
set to 0
the singular values $\sigma_j(M)$,  for $j>r$,
 let $M_r$ denote the resulting matrix, which is a closest rank-$r$
approximation of $M$, and write 
$M=M_r+E.$
Then
%------------------------------------------------------------------------------
%\begin{equation}\label{eqmr} 
$$||E||=\sigma_{r+1}(M)~
{\rm and}~||E||_F^2=\sum_{j=r+1}^n\sigma_j^2\le \sigma_{r+1}(M)^2(n-r).$$
%\end{equation}
\end{lemma}

%------------------------------------------------------------------------------

\begin{theorem}\label{thrrap} {\rm The error norm 
 in terms of $||(M_rB)^+||$}. 
Assume dealing with the
%------------------------------------------------------------------------------
 matrices $M$ and $\tilde M$  
 of 
%\begin{equation}\label{eqrnrra}
Algorithm  \ref{alg1},
$M_r$ and $E$ of Lemma \ref{letrnc}, 
  and $B\in \mathbb C^{n\times l}$ of rank $l$. 
Let $\rank(M_rB)=r$ and write $E'=EB$ and
$\Delta=||\tilde M-M||$.  Then 
%\medskip
\begin{equation}\label{eqe'} 
||E'||_F\le ||B||_F~||E||_F\le ||B||_F~\sigma_{r+1}(M)~ \sqrt{n-r}
\end{equation}
 and
\begin{equation}\label{eqdltsgm}
|\Delta-
\sigma_{r+1}(M)|\le
 \sqrt 8~||(M_rB)^+||~||E'||_F+O(||E'||_F^2).
\end{equation}
 \end{theorem}

%------------------------------------------------------------------------------

\begin{proof}
Lemma \ref{letrnc} implies  bound (\ref{eqe'}). 

Now apply part (ii) of
Theorem \ref{thlrnk} for matrix $M_r$ replacing $M$,
recall  that $\rank(B)=l$, and obtain
$$Q(M_rB)Q(M_rB)^TM_r=M_r,~
\mathcal R(Q(M_rB))=\mathcal R(M_rB)=\mathcal R(M_r).$$ 
Therefore $$Q(M_rB)Q(M_rB)^TM=Q(M_rB)Q(M_rB)^TM_r=M_r.$$

Consequently,  
%with probability 1,
 $M-Q(M_rB)Q(M_rB)^TM=M-M_r=E$, and  so
(cf. Lemma \ref{letrnc})
\begin{equation}\label{eqrrnm0}
||M-Q(M_rB)Q(M_rB)^TM||=\sigma_{r+1}(M).
\end{equation}

%------------------------------------------------------------------------------

\noindent Apply  \cite[Corollary C.1]{PQY15}, for 
$A=M_rB$  and $E$ replaced by $E'=(M-M_r)B$, and obtain 
%\begin{equation}\label{eqrra}
$$||Q(MB)Q(MB)^T-Q(M_rB)Q(M_rB)^T||\le
\sqrt 8||(M_rB)^+||~||E'||_F+O(||E'||_F^2).$$
%\end{equation}

Combine this 
 bound  with
%equation 
(\ref{eqrrnm0}) and
obtain (\ref{eqdltsgm}).
\end{proof}
 
%------------------------------------------------------------------------------

%\subsection{Assessing some immediate implications}\label{simim}

%------------------------------------------------------------------------------

By combining parts (\ref{eqe'}) and (\ref{eqdltsgm}) obtain

%------------------------------------------------------------------------------

\begin{equation}\label{eqdlt}
|\Delta-\sigma_{r+1}(M)|\le 
\sqrt {8(n-r)}~\sigma_{r+1}(M)~||B||_F||(M_rB)^+||+O(\sigma_{r+1}^2(M)).
\end{equation}
In our applications the product $\sigma_{r+1}(M)||B||_F$ is small,  
 and so the value $|\Delta-\sigma_{r+1}(M)|$ 
is  small unless the  norm $||(M_rB)^+||$ is large.

%------------------------------------------------------------------------------

\begin{remark}\label{rernlrpr} {\em  The Power Scheme of increasing the output accuracy 
of Algorithm \ref{alg1}.} See \cite{RST09}, \cite{HMST11}.
% at a low  computational cost.}
Define the Power Iterations
$M_i=(M^TM)^iM$, $i=1,2,\dots$. 
Then $\sigma_j(M_i)=(\sigma_j(M))^{2i+1}$
for all $i$ and $j$  \cite[equation (4.5)]{HMT11}. 
Therefore, at a reasonable computational cost, one can
dramatically decrease the ratio
$\frac{\sigma_{r+1}(M)}{\sigma_r(M)}$ and thus decrease 
 the bounds of Theorems \ref{th1}
and \ref{th1d}  accordingly.
\end{remark}

%------------------------------------------------------------------------------

\subsection{Proof of Theorem \ref{thall}}\label{sprth1}

%------------------------------------------------------------------------------

If $\rank(MB)=\rank(Q)=r_-<r$, then $\rank(\tilde M)\le r_-<r$,
$\Delta\ge \sigma_{r_-}(M)$, that is, not small since $\nrank(M)=r>r_-$,
and so Algorithm \ref{alg1}
applied to $M$ with the multiplier $B$ outputs FAILURE.
If $\rank(MB)=r>\nrank(MB)=r_-$, 
then $\rank(MB-E)=r_-<r$ for a small-norm perturbation matrix $E$.
Hence $\Delta\ge \sigma_{r_-}(M)-O(||E||)$,
and then again Algorithm \ref{alg1}
applied to $M$ with the multiplier $B$ outputs FAILURE.
This proves the ``only if" part of the claim of Theorem \ref{thall}.

Now let $\nrank(MB)=r$ and assume that we scaled the matrix $B$
so that $||B||_F=1$.
Then $\rank(MB)=r$ (and so we can apply bound (\ref{eqdlt})),
and furthermore $\nrank(M_rB)=\nrank(MB)=r$.
Equation (\ref{eqdlt}) implies that 
$\Delta\approx 8\sqrt{8(n-r)}\sigma_{r+1}||(M_rB)^+||$,
which is a small positive value because $\nrank (M)=r$,
and so the value $|\sigma_{r+1}|$ is small,
and part ``if" of Theorem \ref{thall} follows.

%------------------------------------------------------------------------------

\subsection{Detailed estimates for primal and dual low-rank approximation}\label{slra}

%------------------------------------------------------------------------------

The following theorem, proven in the next subsection,   
bounds  
the approximation errors and the 
probability of success of Algorithm \ref{alg1} for $B\in \mathcal G^{n\times l}$.
Together these bounds imply part (ii) of Theorem \ref{th0}. 

\begin{theorem}\label{th1} 
%{\rm The Power of Gaussian random sampling.} 
Suppose that  Algorithm \ref{alg1}  has been applied to 
an $m\times n$ matrix $M$ having numerical rank $r$ and
 that the multiplier $B=G_{n,l}$ is an $n\times l$ Gaussian matrix.

%------------------------------------------------------------------------------

(i) Then the algorithm outputs an approximation $\tilde M$ of a matrix $M$ 
 by a rank-$l$ matrix within the  error norm bound
  $\Delta$ such that 
$|\Delta-\sigma_{r+1}(M)|\le f\sigma_{r+1}(M)+O(\sigma_{r+1}^2(M))$ where 
$f=\sqrt {8(n-r)}~\nu_{F,n,l}\nu^+_{r,l}/\sigma_r(M)$
and
$\nu_{F,n,l}$ and $\nu^+_{r,l}$
are random variables 
of Definition \ref{defnrm}.

%------------------------------------------------------------------------------

(ii) $\mathbb E(f)<\frac{1+\sqrt n+\sqrt l}{p\sigma_r(M)}~e~\sqrt{8(n-r)rl}$, 
for  $p=l-r>0$  and  $e=2.71828\dots$.
\end{theorem}

% - - - - - - - - - - - - - - - - - - - - - - - - - - - - - - - - - - - - -

\begin{remark}\label{reopt}
$\sigma_{r+1}(M)$ is the optimal upper bound on the norm $\Delta$, and
the expected value $\mathbb E(f)$ 
is reasonably small even for $p=1$. 
If $p=0$, then $\mathbb E(f)$  is not defined, 
%that is, if  we apply no oversampling,
 but 
the random variable $\Delta$ estimated in Theorem
\ref{th1} is still likely to be reasonably close
to $\sigma_{r+1}(M)$ (cf. part (ii) of Theorem  \ref{thsiguna}).
\end{remark}

%------------------------------------------------------------------------------

In Section \ref{sdlrnd} we  prove the following elaboration upon dual Theorem  \ref{th0d};
its upper bound on the output error norm of Algorithm \ref{alg1}
is a little smaller than the bound of primal 
 Theorem  \ref{th1}.

\begin{theorem}\label{th1d}
Suppose that Algorithm \ref{alg1}, applied to
  a small-norm perturbation 
of 
an $m\times n$
 fac\-tor-Gauss\-ian matrix with expected 
 rank $r<m$, uses  
an 
$n\times l$ multiplier $B$ such that $\nrank(B)=l$ and
$l\ge r$. 

%------------------------------------------------------------------------------

(i) Then
 the algorithm outputs  
  a rank-$l$ matrix  $\tilde M$ that
 approximates the matrix $M$ within the  error norm bound
  $\Delta$ such that $|\Delta-\sigma_{r+1}(M)|\le f_d\sigma_{r+1}(M)+O(\sigma_{r+1}^2(M))$, where   
$f_d=\sqrt {8(n-r)l}~\nu^+_{r,l}\nu_{m,r}^+\kappa(B)$,
 $\kappa(B)=||B||~||B^+||$,
% denotes the condition number of the matrix $B$ 
and  
 $\nu_{m,r}^+$ and $\nu^+_{r,l}$
are random variables 
of Definition \ref{defnrm}.

%------------------------------------------------------------------------------

(ii) $\mathbb E(f_d)<e^2\sqrt{8(n-r)l}~\kappa(B)\frac{r}{(m-r)p}$, 
for  $p=l-r>0$  and  $e=2.71828\dots$.
\end{theorem}

%------------------------------------------------------------------------------

\begin{remark}\label{redl}
The expected value $\mathbb E(\nu_{m,r}^+)=\frac{e\sqrt r}{m-r}$ converges to 0 as $m\rightarrow \infty$
provided that $r\ll m$. 
Consequently the expected value $\mathbb E(\Delta)=\sigma_{r+1}(M)\mathbb E(f_d)$ 
converges to the optimal value $\sigma_{r+1}(M)$
as $\frac{m}{r\sqrt {nl}}\rightarrow \infty$ provided 
that $B$ is a well-conditioned matrix
of full rank and that $1\le r<l\ll n\le m$. 
\end{remark}

%------------------------------------------------------------------------------

\begin{remark}\label{repfr} 
 \cite[Theorem 10.8]{HMT11} also 
estimates the norm $\Delta$, but our estimate in Theorem \ref{th1}, 
in terms of random variables $\nu_{F,n,l}$ and $\nu^+_{r,l}$,
is more compact, and 
our proof is distinct and shorter than one
 in \cite{HMT11}, which involves the proofs of  \cite[Theorems 9.1, 10.4 and 10.6]{HMT11}.
\end{remark}

\begin{remark}\label{reprob}  
By virtue of 
Theorems \ref{thrnd}, $\rank(M_rB)=r$
with probability 1 if the matrix $B$ or $M$ is Gaussian,
which is the case of Theorems \ref{th1} and \ref{th1d},
and under the equation  $\rank(M_rB)=r$ we proved bound (\ref{eqdlt}).
\end{remark}

In the next two subsections
 we deduce reasonable bounds on the  norm $||(M_rB)^+||$ in both cases where 
 $M$ is a fixed  matrix and $B$ is a Gaussian matrix 
and where $B$ is fixed  matrix and $M$ is a factor Gaussian matrix 
(cf. Theorems \ref{thprm} and \ref{thdual}).
The bounds imply Theorems \ref{th1} and \ref{th1d}.

%------------------------------------------------------------------------------

\subsection{Primal theorem: completion of the proof}\label{sprrnd}

%------------------------------------------------------------------------------

%We use the auxiliary results of the Appendix and the following simple lemma.
%(cf. items 1.4.7 and 1.4.10).

%------------------------------------------------------------------------------

%\begin{lemma}\label{lepr1} 
%Suppose that  $H$ is an $n\times r$ matrix,
%$\Sigma=\diag(\sigma_i)_{i=1}^{n}$, 
%$\sigma_1\ge \sigma_2\ge \cdots \ge \sigma_n>0$, 
%$\Sigma'=\diag(\sigma_i')_{i=1}^{r}$,
% $\sigma'_1\ge \sigma'_2\ge \cdots \ge \sigma'_r>0$.
%Then 
%$\sigma_{j}(\Sigma H\Sigma')\ge\sigma_{j}(H)\sigma_n \sigma'_r~{\rm for~all}~j$.
%\end{lemma}

%------------------------------------------------------------------------------

\begin{theorem}\label{thprm} 
For $M\in \mathbb R^{m\times n}$,
 $B\in \mathcal G^{m\times l}$,
and $\nu_{r,l}^+$ 
of Definition \ref{defnrm},
it holds that
%(cf. item 1.4.9)
  \begin{equation}\label{eqn+}
||(M_rB)^+||\le \nu_{r,l}^+/\sigma_r(M). 
\end{equation}
\end{theorem}

\begin{proof}
Let $M_r=S_r\Sigma_rT_r^T$ be compact SVD.

By applying  Lemma \ref{lepr3}, deduce that $T_r^TB$
is a $r\times l$ Gaussian matrix. 

Denote it $G_{r,l}$ 
and obtain
$M_rB=S_r\Sigma_rT_r^TB=S_r\Sigma_rG_{r,l}$.

Write $H=\Sigma_rG_{r,l}$ and let $H=S_H\Sigma_HT_H^T$ be compact  SVD
where $S_H$ is a $r\times r$ unitary matrix.

It follows that $S=S_rS_H$ is an $m\times r$ unitary matrix.

Hence
$M_rB=S\Sigma_HT_H^T$ and $(M_rB)^+=T_H(\Sigma_H)^+S^T$
are compact SVDs of the matrices $M_rB$ and $(M_rB)^+$, respectively.

Therefore 
$||(M_rB)^+||=||(\Sigma_H)^+||=
||(\Sigma_rG_{r,l})^+||\le ||G_{r,l}^+||~||\Sigma_r^{-1}||$.

Substitute $||G_{r,l}^+||=\nu_{r,l}^+$ and $||\Sigma_r^{-1}||=1/\sigma_r(M)$
and obtain the theorem.
\end{proof}

%------------------------------------------------------------------------------
  
Combine bounds (\ref{eqdlt}), (\ref{eqn+}),  and equation $||B||_F=\nu_{F,n,l}$
and obtain part (i) of Theorem \ref{th1}.
Combine that part with parts (ii) of Theorem \ref{thsignorm} and
(iii) of  Theorem \ref{thsiguna} 
 and obtain part (ii) of Theorem \ref{th1}.

%------------------------------------------------------------------------------
%------------------------------------------------------------------------------

\subsection{ Dual  theorem: completion of the proof}\label{sdlrnd}

%------------------------------------------------------------------------------

\begin{theorem}\label{thdual}
Suppose that   
$U\in \mathbb R^{m\times r}$,
 $V\in \mathcal G^{r\times n}$, $\rank(U)=r\le \min\{m,n\}$,
$M=UV$,
and $B$ is 
a well-conditioned $n\times l$ matrix 
of full rank $l$ such that $m\ge n>l\ge r$
and $||B||_F=1$. 
Then 
 \begin{equation}\label{eqdlt1}
 ||(MB)^+|| \le ||U^+||~\nu_{r,l}^+~||B^+||.
\end{equation}
If in addition  $U\in \mathcal G^{m\times r}$, that is, if
$M$   is an $m\times n$ fac\-tor-Gauss\-ian matrix with expected rank $r$, then
 \begin{equation}\label{eqmb+}
||(MB)^+|| \le \nu_{m,r}^+~\nu_{r,l}^+~||B^+||.
\end{equation}
\end{theorem}

\begin{proof}
Combine compact
SVDs  
$U=S_{U} \Sigma_{U}T_{U}^T$ 
and  $B=S_{B} \Sigma_{B}T_{B}^T$
%where  
%$S_{U,r}$ and $T_{U,r}$
%denote the $m\times r$
%and $n\times r$ matrices made up of the first 
%$r$ columns of the orthogonal matrices $S_U$ and  $T_U$, 
%respectively, and
%$\Sigma_{U,r}$ is the $r\times r$ diagonal matrix of  
%the $r$ largest singular values of  the matrix $U$.
and obtain $UVB=
S_{U} \Sigma_{U}T^T_{U}VS_B\Sigma_{B}T_{B}^T$.
Here $U$, $V$, $B$,  $S_U$, $\Sigma_U$,  $T_U$, $S_B$, 
$\Sigma_B$, and $T_B$ are matrices of the sizes $m\times r$, 
$r\times n$, $n\times l$, $m\times r$, $r\times r$, $r\times r$, 
$n\times l$, $l\times l$, and $l\times l$, respectively.

Now observe that $G_{r,l}=T_{U}^TVS_{B}$ is a $r\times l$
Gaussian matrix,
by  virtue of Lemma \ref{lepr3} (since $V=G_{r,n}$ is a Gaussian matrix).
Therefore $UVB=S_{U} FT_{B}^T$, for $F= \Sigma_{U}G_{r,l}\Sigma_{B}$.
%and  ortho\-gonal matrices $S_{U}$  of size $m\times r$
%and $T_{B}$ of size $l\times l$.  

Let $F=S_F\Sigma_F T_F^T$ denote compact SVD
where $\Sigma_F=\diag (\sigma_j(F))_{j=1}^r$
and $S_F$ and $T_F^T$ are unitary matrices of sizes 
$r\times r$ and $r\times l$,  respectively.

Both products $S_{U}S_F\in \mathbb R^{m\times r}$ and 
$T_F^TT_{B}^T\in \mathbb R^{r\times l}$ are unitary matrices,
 and we obtain compact SVD 
$MB=UVB=S_{MB}\Sigma_{MB}T^T_{MB}$
where $S_{MB}=S_{U}S_F$, $\Sigma_{MB}=\Sigma_F$, 
and $T^T_{MB}=T_F^TT_{B}^T$.
Therefore
  $$||(MB)^+||=||\Sigma^+_{MB}||=||\Sigma^+_F||= ||F^+||.$$
Note that $F^+=\Sigma_B^{-1}G_{r,l}^+\Sigma_U^{-1}$
because $\Sigma_B$ and $\Sigma_V$ are square nonsingular diagonal matrices.
It follows that
$$||F^+||\le ||\Sigma_B^{-1}||~||G_{r,l}^+||~||\Sigma_U^{-1}||=||B^+||\nu_{r,l}^+||V^+||.$$
Consequently
$||(MB)^+||=||\Sigma^+_F||\le ||U^+||\nu_{r,l}^+||B^+||$. Now (\ref{eqdlt1}) follows
and implies (\ref{eqmb+}).
 \end{proof}

Combine (\ref{eqdlt}), (\ref{eqdlt1})
and $||B||_F\le ||B||~\sqrt l$ and  obtain
 Theorem \ref{th1d} provided that
$M$ is a fac\-tor-Gauss\-ian matrix $UV$ with expected rank $r$.
Apply Theorem \ref{thpert} to extend the results 
to the case where  
$M=UV+E$ and the norm $||E||$ is small, completing the proof of 
Theorem \ref{th1d}.

%------------------------------------------------------------------------------

\begin{remark}\label{regnrl} 
If $U\in \mathcal G^{m\times r}$, for $m-r\ge 4$, then
it is likely that
$\nrank(U)=r$ by virtue of Theorem \ref{thsiguna},
and our proof of bound (\ref{eqdlt1}) applies even if  we assume that
$\nrank(U)=r$ rather than $U\in \mathcal G^{m\times r}$.
\end{remark}

%------------------------------------------------------------------------------

\section{Numerical Tests}\label{ststs}

% - - - - - - - - - - - - - - - - - - - - - - - - - - - - - - - - - - - - -

 Numerical experiments have been 
%designed by the first author and have been 
 performed by  Xiaodong Yan for Tables  
\ref{tab67}--\ref{tab614}
and by John Svadlenka and Liang Zhao for the other tables.
The tests have been run by using MATLAB  
 in the Graduate Center of the City University of New York 
on a Dell computer with the Intel Core 2 2.50 GHz processor and 4G memory running 
Windows 7, 
and in particular the standard normal distribution function randn of MATLAB
has been applied in order to generate Gaussian matrices.

We calculated the $\xi$-rank, i.e., the number of singular values 
exceeding $\xi$, by applying the MATLAB function "svd()". 
We have set $\xi=10^{-5}$ in Sections \ref{ststssvd} and
 \ref{ststslo} and  $\xi=10^{-6}$
in Section \ref{s17m}.

% - - - - - - - - - - - - - - - - - - - - - - - - - - - - - - - - - - - - -

\subsection{Tests for inputs generated via SVD}\label{ststssvd}

% - - - - - - - - - - - - - - - - - - - - - - - - - - - - - - - - - - - - -

In the tests of this subsection we generated $n\times n$ input matrices $M$  
 by extending the customary recipes of [H02, Section 28.3]. Namely, we first  
generated matrices $S_M$ and $T_M$ 
by means of the orthogonalization of  
$n\times n$ Gaussian matrices. Then we defined
$n\times n$ matrices $M$ by 
their compact SVDs, $M=S_M\Sigma_M T_M^T$,
for $\Sigma_M=\diag(\sigma_j)_{j=1}^n$; 
 $\sigma_j=1/j,~j=1,\dots,r$,
$\sigma_j=10^{-10},~j=r+1,\dots,n$, 
 and  $n=256,
512,
1024$.
(Hence $||M||=1$ and 
$\kappa(M)=||M||~||M^{-1}||=10^{10}$.) 

Table \ref{tab1} shows
the average output error norm $\Delta$ 
over  1000 tests of  Algorithm \ref{alg1} applied to 
these matrices $M$ for 
each pair of $n$ and $r$,
  $n=256,
512,
1024$, $r=8,32$, and
each of the following three groups of multipliers:
%\begin{enumerate}
%\item%1 
3-AH multipliers, 
%\item%2
3-ASPH  multipliers, both
 defined by 
 Hadamard recursion (\ref{eqfd}),  
 for $d=3$, and 
%\item%3
dense  multipliers $B=B(\pm 1,0)$  
having i.i.d. entries $\pm 1$ and 0,
each value chosen with probability 1/3.
%\end{enumerate}  

%------------------------------------------------------------------------------

\begin{table}[ht] 
  \caption{Error norms for SVD-generated inputs 
and 3-AH, 3-ASPH,  and $B(\pm 1,0)$ multipliers}
\label{tab1}

  \begin{center}
    \begin{tabular}{|*{8}{c|}}
      \hline
$n$ & $r$  &3-AH&3-ASPH& $B(\pm 1,0)$ %&4-BRI 
\\ \hline
256 & 8 & 2.25e-08 & 2.70e-08 & 2.52e-08 
%& 2.49e-08  
\\\hline
256 & 32 &5.95e-08 & 1.47e-07 & 3.19e-08 
%& 4.38e-08 
\\\hline
512 & 8 &4.80e-08 & 2.22e-07 & 4.76e-08 
%& 1.71e-07 
\\\hline
512 & 32 & 6.22e-08 & 8.91e-08 & 6.39e-08 
%& 6.10e-08
\\\hline
1024 & 8 & 5.65e-08 & 2.86e-08 & 1.25e-08
%& 1.45e-07
\\\hline
1024 & 32 & 1.94e-07 & 5.33e-08 & 4.72e-08
% & 1.15e-08
\\\hline
    \end{tabular}
  \end{center}
\end{table}

% - - - - - - - - - - - - - - - - - - - - - - - - - - - - - - - - - - - - -

Tables \ref{tab67}--\ref{tab614} show  the  mean and maximal values 
 of such an error norm in the case of   
(a) real Gaussian multipliers $B$ and dense real Gaussian subcirculant multipliers $B$
of  Section \ref{sdfcndf}, for $q=n$, each  defined by its first column filled with 
either (b) i.i.d.  Gaussian
variables or (c) random variables
  $\pm 1$. 
Here 
and hereafter in this section we assigned each random 
signs $+$ or $-$ with probability 0.5.

%------------------------------------------------------------------------------

\begin{table}[ht] 
  \caption{Error norms for SVD-generated inputs and
Gaussian multipliers}
\label{tab67}
  \begin{center}
    \begin{tabular}{| c |  c | c |  c |c|}
      \hline
$r$   & $n$ & \bf{mean} & \bf{max} \\ \hline	
8 & 256 & $ 7.54\times 10^{-8} $  &  $ 1.75\times 10^{-5} $    \\ \hline		
8 & 512 &  $ 4.57\times 10^{-8} $  &  $ 5.88\times 10^{-6} $  \\ \hline		
8 & 1024 &  $ 1.03\times 10^{-7} $  &  $ 3.93\times 10^{-5} $    \\ \hline		
32 & 256 & $ 5.41\times 10^{-8} $  &  $ 3.52\times 10^{-6} $     \\ \hline
32 & 512 & $ 1.75\times 10^{-7} $  &  $ 5.57\times 10^{-5} $   \\ \hline		
32 & 1024 & $ 1.79\times 10^{-7} $  &  $ 3.36\times 10^{-5} $    \\ \hline	
    \end{tabular}
  \end{center}
\end{table}

%------------------------------------------------------------------------------

\begin{table}[ht] 
  \caption{Error norms  for SVD-generated inputs  and
Gaussian  subcirculant multipliers}
\label{tab611}
  \begin{center}
    \begin{tabular}{| c | c | c | c | c | c |c|}     \hline
$r$  &  $n$ & \bf{mean} & \bf{max}  \\ \hline
8 & 256 & $ 3.24\times 10^{-8} $  &  $ 2.66\times 10^{-6} $   \\ \hline		
8 & 512 &  $ 5.58\times 10^{-8} $  &  $ 1.14\times 10^{-5} $  \\ \hline		
8 & 1024 &  $ 1.03\times 10^{-7} $  &  $ 1.22\times 10^{-5} $  \\ \hline		
32 & 256 & $ 1.12\times 10^{-7} $  &  $ 3.42\times 10^{-5} $   \\ \hline
32 & 512 & $ 1.38\times 10^{-7} $  &  $ 3.87\times 10^{-5} $   \\ \hline		
32 & 1024 & $ 1.18\times 10^{-7} $  &  $ 1.84\times 10^{-5} $  \\ \hline	
    \end{tabular}
  \end{center}
\end{table}

%------------------------------------------------------------------------------

\begin{table}[ht] 
  \caption{Error norms for SVD-generated inputs and
random subcirculant
  multipliers  filled with $\pm 1$}
\label{tab614}
  \begin{center}
    \begin{tabular}{| c | c | c | c | c | c |c|}     \hline
$r$  &  $n$ & \bf{mean} & \bf{max} \\ \hline			
8 & 256 & $ 7.70\times 10^{-9} $  &  $ 2.21\times 10^{-7} $   \\ \hline		
8 & 512 &  $ 1.10\times 10^{-8} $  &  $ 2.21\times 10^{-7} $  \\ \hline		
8 & 1024 &  $ 1.69\times 10^{-8} $  &  $ 4.15\times 10^{-7} $   \\ \hline		
32 & 256 & $ 1.51\times 10^{-8} $  &  $ 3.05\times 10^{-7} $   \\ \hline
32 & 512 & $ 2.11\times 10^{-8} $  &  $ 3.60\times 10^{-7} $   \\ \hline		
32 & 1024 & $ 3.21\times 10^{-8} $  &  $ 5.61\times 10^{-7} $  \\ \hline		
    \end{tabular}
  \end{center}
\end{table}

%------------------------------------------------------------------------------
 
%\clearpage

Table \ref{LowRkEx} displays  the
average error norms in
the case of multipliers $B$
of eight kinds defined below, all  
generated from the following  Basic Sets 1, 2 and 3
of $n\times n$ multipliers:

{\em Basic Set 1}:  3-APF
multipliers defined by 
three Fourier recursive steps of
 equation (\ref{eqfd}), for $d=3$,
with no scaling, but  
with a random column permutation.

{\em Basic Set 2}: Sparse real circulant matrices $Z_1({}\bf v)$
of family (ii) of  Section \ref{sdfcndf}, for $q=10$, 
whose  first column vectors  ${\bf v}$ have been filled with zeros,
except for  ten random coordinates filled with random integers $\pm 1$.

{\em Basic Set 3}:   Sum of two scaled inverse bidiagonal matrices. 
We first filled the main diagonals of both matrices with the integer 101
 and  their first subdiagonals 
 with $\pm 1$. Then
we  multiplied  each matrix by a  diagonal 
matrix $\diag(\pm 2^{b_i})$, where $b_i$ were random integers
uniformly chosen from 0 to 3.

For multipliers $B$ we used the $n\times r$  western 
(leftmost) blocks of $n\times n$ matrices
of the following classes:
\begin{enumerate}
\item%1
  a matrix from Basic Set 1; 
\item%2
  a matrix from Basic Set 2;
\item%3
 a matrix from Basic Set 3;
\item%4
 the product of two matrices of Basic Set 1;
\item%5
 the product of two matrices of Basic Set 2;
\item%6
 the product of two matrices of Basic Set 3;
\item%7
 the sum of two matrices of Basic Sets 1 and 3,
and 
\item%8
 the sum of two matrices of Basic Sets 2 and 3.
\end{enumerate}
The tests
produced the results similar to the ones of Tables \ref{tab1}--\ref{tab614}.
 
In sum, for all classes of input  pairs $M$ and $B$ and all pairs of integers $n$ and $r$,
Algorithm \ref{alg1} with our preprocessing 
has consistently output approximations to rank-$r$ input matrices with  the
average error norms 
 ranged from $10^{-7}$ or $10^{-8}$ to about $10^{-9}$
in all our tests.

\begin{table}[ht] 
  \caption{Error norms  for SVD-generated inputs  and
multipliers of eight classes}
\label{LowRkEx}
  \begin{center}
    \begin{tabular}{| c |  c | c |  c |c|c|c|c|c|c|}
      \hline
$n$ & $r$ &class 1 &class 2 &class 3 &class 4 &class 5 &class 6 &class 7 &class 8 \\\hline
256 & 8 &5.94e-09 &4.35e-08 &2.64e-08 &2.20e-08 &7.73e-07 &5.15e-09 &4.08e-09 &2.10e-09  \\\hline
256 & 32 &2.40e-08 &2.55e-09 &8.23e-08 &1.58e-08 &4.58e-09 &1.36e-08 &2.26e-09 &8.83e-09 \\\hline 
512 & 8 &1.11e-08 &8.01e-09 &2.36e-09 &7.48e-09 &1.53e-08 &8.15e-09 &1.39e-08 &3.86e-09  \\\hline
512 & 32 &1.61e-08 &4.81e-09 &1.61e-08 &2.83e-09 &2.35e-08 &3.48e-08 &2.25e-08 &1.67e-08\\\hline 
1024 & 8 &5.40e-09 &3.44e-09 &6.82e-08 &4.39e-08 &1.20e-08 &4.44e-09 &2.68e-09 &4.30e-09 \\\hline 
1024 & 32 &2.18e-08 &2.03e-08 &8.72e-08 &2.77e-08 &3.15e-08 &7.99e-09 &9.64e-09 &1.49e-08\\\hline 
    \end{tabular}
  \end{center}
\end{table}

%------------------------------------------------------------------------------
%------------------------------------------------------------------------------

  We summarize the  results of the tests of this subsection for $n=1024$ and $r=8,32$
in Figure \ref{LowRkTest1}.

\begin{figure}[htb] 
\centering
\caption{Error norms in the tests of Section \ref{ststssvd}}
\label{LowRkTest1}
\end{figure}

%In the tests  average error norm ranged 
%from $10^{-7}$ to $10^{-9}$.
 
%\clearpage

%In the tests  average error norm ranged 
%from $10^{-7}$ to $10^{-9}$.
 
%\clearpage

% - - - - - - - - - - - - - - - - - - - - - - - - - - - - - - - - - - - - -

\subsection{Tests for inputs generated via the discretization of a Laplacian operator
and via the approximation of an inverse finite-difference operator}\label{ststslo}

% - - - - - - - - - - - - - - - - - - - - - - - - - - - - - - - - - - - - -

Next we present the test results for Algorithm \ref{alg1} applied
 to input matrices for computational problems of two kinds,
both replicated from  \cite{HMT11}, namely, the matrices of
 
(i) the discretized single-layer Laplacian operator and 

(ii) the approximation of the inverse of a finite-difference operator.

{\em Input matrices (i).} We considered the Laplacian operator
%\begin{equation}
$[S\sigma](x) = c\int_{\Gamma_1}\log{|x-y|}\sigma(y)dy,x\in\Gamma_2$,
%\end{equation}
from  \cite[Section 7.1]{HMT11},
for two contours $\Gamma_1 = C(0,1)$ and $\Gamma_2 = C(0,2)$  on the complex plane.
Its dscretization defines an $n\times n$ matrix $M=(m_{ij})_{i,j=1}^n$  
where
%\begin{equation}
$m_{i,j} = c\int_{\Gamma_{1,j}}\log|2\omega^i-y|dy$
for a constant $c$ such that  $||M||=1$ and
%\end{equation} 
for the arc $\Gamma_{1,j}$  of the contour $\Gamma_1$ defined by
the angles in $[\frac{2j\pi}{n},\frac{2(j+1)\pi}{n}]$.
We applied Algorithm \ref{alg1}
supported by  
 three iterations of the Power Scheme of Remark  \ref{rernlrpr} 
and used with multipliers
 $B$ being the $n\times r$ leftmost submatrices of $n\times n$
matrices of 
the following five  classes: 
\begin{itemize}
\item%1
 Gaussian  multipliers, 
\item%2
 Gaussian Toeplitz  multipliers $T=(t_{i-j})_{i=0}^{n-1}$ 
for i.i.d. Gaussian variables $t_{1-n},\dots,t_{-1}$,$t_0,t_1,\dots,t_{n-1}$.
  \item%3
 Gaussian circulant  multipliers $\sum_{i=0}^{n-1}v_iZ_1^i$, 
for i.i.d. Gaussian variables $v_0,\dots,v_{n-1}$ and the unit circular matrix $Z_1$
of Section \ref{sdfcnd}.  
\item%4
 Abridged permuted  Fourier (3-APF) multipliers, and 
\item%5
 Abridged permuted Hadamard (3-APH) multipliers.
\end{itemize}

As in the previous subsection,
we defined each 
%Abridged Permuted Fourier or Hadamard
 3-APF and 3-APH  matrix by applying
three recursive steps of equation (\ref{eqfd}) followed
by a single random column permutation.

%We have increased the output accuracy by 
%applying the Power Scheme of Remark \ref{rernlrpr}.
 
We applied Algorithm \ref{alg1} with  multipliers of all five listed classes.
For each setting we repeated the test 1000 times and calculated the mean and standard deviation of the error norm $||\tilde M - M||$. 
%Table \ref{ExpHMT1} displays test results for $n$ up to 4000.
  
\begin{table}[h]
\caption{Low-rank approximation  of Laplacian  matrices}
\label{ExpHMT1}
\begin{center}
\begin{tabular}{|*{5}{c|}}
\hline
$n$ 	& multiplier 	& $r$	& mean	& std\\\hline
200 & Gaussian &  3.00 & 1.58e-05 & 1.24e-05\\\hline
200 & Toeplitz &  3.00 & 1.83e-05 & 7.05e-06\\\hline
200 & Circulant &  3.00 & 3.14e-05 & 2.30e-05\\\hline
200 & 3-APF &  3.00 & 8.50e-06 & 5.15e-15\\\hline
200 & 3-APH &  3.00 & 2.18e-05 & 6.48e-14\\\hline
400 & Gaussian &  3.00 & 1.53e-05 & 1.37e-06\\\hline
400 & Toeplitz &  3.00 & 1.82e-05 & 1.59e-05\\\hline
400 & Circulant &  3.00 & 4.37e-05 & 3.94e-05\\\hline
400 & 3-APF &  3.00 & 8.33e-06 & 1.02e-14\\\hline
400 & 3-APH &  3.00 & 2.18e-05 & 9.08e-14\\\hline
2000 & Gaussian &  3.00 & 2.10e-05 & 2.28e-05\\\hline
2000 & Toeplitz &  3.00 & 2.02e-05 & 1.42e-05\\\hline
2000 & Circulant &  3.00 & 6.23e-05 & 7.62e-05\\\hline
2000 & 3-APF &  3.00 & 1.31e-05 & 6.16e-14\\\hline
2000 & 3-APH &  3.00 & 2.11e-05 & 4.49e-12\\\hline
4000 & Gaussian &  3.00 & 2.18e-05 & 3.17e-05\\\hline
4000 & Toeplitz &  3.00 & 2.52e-05 & 3.64e-05\\\hline
4000 & Circulant &  3.00 & 8.98e-05 & 8.27e-05\\\hline
4000 & 3-APF &  3.00 & 5.69e-05 & 1.28e-13\\\hline
4000 & 3-APH &  3.00 & 3.17e-05 & 8.64e-12\\\hline
\end{tabular}
\end{center}
\end{table}

{\em Input matrices (ii).} We similarly applied Algorithm \ref{alg1} to the input matrix $M$ 
being the inverse of a large sparse matrix obtained from a finite-difference operator
 from  \cite[Section 7.2]{HMT11} 
and observed  similar results
with all structured  and Gaussian multipliers.

We performed 1000 tests for every class of pairs of $n\times n$ or $m\times n$ matrices 
of classes (i) or (ii), respectively,
and 
$n\times r$ multipliers for every fixed triple of $m$, $n$, and $r$ or pair of $n$ and $r$.

Tables \ref{ExpHMT1} and \ref{ExpHMT2} display the resulting data for the mean values and standard deviation of the error norms, and we summarize the results of the tests of this subsection 
in Figure \ref{LowRkTest2}.

\begin{table}[h]
\caption{Low-rank approximation of the matrices of discretized finite-difference operator}
\label{ExpHMT2}
\begin{center}
\begin{tabular}{|*{6}{c|}}
\hline
$m$ 	& $n$ 	& multiplier 	& $r$	& mean	& std\\\hline
88 & 160 & Gaussian &  5.00 & 1.53e-05 & 1.03e-05\\\hline
88 & 160 & Toeplitz &  5.00 & 1.37e-05 & 1.17e-05\\\hline
88 & 160 & Circulant &  5.00 & 2.79e-05 & 2.33e-05\\\hline
88 & 160 & 3-APF &  5.00 & 4.84e-04 & 2.94e-14\\\hline
88 & 160 & 3-APH &  5.00 & 4.84e-04 & 5.76e-14\\\hline
208 & 400 & Gaussian & 43.00 & 4.02e-05 & 1.05e-05\\\hline
208 & 400 & Toeplitz & 43.00 & 8.19e-05 & 1.63e-05\\\hline
208 & 400 & Circulant & 43.00 & 8.72e-05 & 2.09e-05\\\hline
208 & 400 & 3-APF & 43.00 & 1.24e-04 & 2.40e-13\\\hline
208 & 400 & 3-APH & 43.00 & 1.29e-04 & 4.62e-13\\\hline
408 & 800 & Gaussian & 64.00 & 6.09e-05 & 1.75e-05\\\hline
408 & 800 & Toeplitz & 64.00 & 1.07e-04 & 2.67e-05\\\hline
408 & 800 & Circulant & 64.00 & 1.04e-04 & 2.67e-05\\\hline
408 & 800 & 3-APF & 64.00 & 1.84e-04 & 6.42e-12\\\hline
408 & 800 & 3-APH & 64.00 & 1.38e-04 & 8.65e-12\\\hline
\end{tabular}
\end{center}
\end{table} 

%------------------------------------------------------------------------------

\begin{figure}[htb] 
\centering
\caption{Error norms in the tests of Section \ref{ststssvd}}
\label{LowRkTest2}
\end{figure}

%\clearpage 

\subsection{Tests with additional classes of multipliers}\label{s17m} 

In this subsection we display the mean values and standard deviations
of the  error norms observed 
when we repeated the tests of the two previous subsections 
for the same three classes of input matrices
 (that is, SVD-generated, Laplacian, and matrices obtained by discretization of 
 finite difference operators), but now we applied Algorithm \ref{alg1} with
  seventeen  additional classes of multipliers (besides its control application with
 Gaussian multipliers). 
 
We tested  Algorithm \ref{alg1} applied to $1024\times 1024$ SVD-generated input matrices having numerical nullity $r = 32$, to $400 \times 400$ Laplacian input matrices
having numerical nullity $r = 3$, 
and
to $408 \times 800$ matrices having numerical nullity $r = 64$ and
representing finite-difference inputs. 

Then again we repeated the tests 1000 times for each class of input matrices and each 
size of an input and a multiplier, and we display the resulting average error norms 
in Table \ref{SuperfastTable} and Figures \ref{SuperfastSVD}--\ref{SuperfastFD}.

We used multipliers defined as the seventeen sums of $n\times r$ matrices
of the following basic families:

\begin{itemize}
  \item%1
  3-ASPH matrices
\item%2
  3-APH matrices
\item%3
  Inverses of bidiagonal matrices 
\item%4
 Random permutation matrices
\end{itemize}

Here every 3-APH matrix has been defined by three Hadamard's recursive steps
(\ref{eqrfd}) followed by random permutation.
Every 3-ASPH matrix has been defined similarly, but also random scaling has 
been applied with a diagonal matrix $D=\diag(d_i)_{i=1}^n$
having the values of random i.i.d. variables $d_i$ uniformly chosen from the set
$\{1/4,1/2,1,2,4\}$.
 
We permuted all inverses of bidiagonal matrices except for Class 5 of multipliers.

Describing our multipliers we use the following acronyms and abbreviations:
``IBD" for ``the inverse of a bidiagonal",
``MD" for ``the main diagonal", ``SB" for ``subdiagonal", and ``SP" for ``superdiagonal".
We write ``MD$i$", ``$k$th SB$i$" and ``$k$th SP$i$" in order to denote
that the main diagonal, the $k$th subdiagonal, or  the $k$th superdiagonal 
of a bidiagonal matrix, respectively,
was filled with the integer $i$.

\begin{itemize}

\item Class 0:	Gaussian
\item Class 1:	Sum of a 3-ASPH  and two IBD matrices: \\
	B1 with MD$-1$  and  2nd SB$-1$ and
	B2 with MD$+1$ and 1st SP$+1$
\item Class 2:	Sum of a 3-ASPH  and two IBD matrices: \\
        B1 with MD$+1$ and 2nd SB$-1$  and  
        B2 with MD$+1$ and 1st SP$-1$
\item Class 3:	Sum of   a 3-ASPH  and two IBD matrices: \\
	B1 with MD$+1$ and  1st SB$-1$ and
	B2 with MD $+1$  and 1st SP$-1$ 
\item Class 4:	Sum of  a 3-ASPH  and two IBD matrices: \\
	 B1 with MD$+1$ and 1st SB$+1$ and
        B2 with MD$+1$ and 1st SP$-1$
\item Class 5:	Sum of  a 3-ASPH  and two IBD matrices:  \\
	B1 with MD$+1$ and 1st SB$+1$ and B2 with MD$+1$ and 1st SP$-1$
\item Class 6:	Sum of a 3-ASPH  and three IBD matrices:\\
	B1 with MD$-1$ and  2nd SB$-1$,
	B2 with MD$+1$ and 1st SP$+1$ and
	B3 with MD$+1$ and 9th SB$+1$
\item Class 7:	Sum of a 3-ASPH  and three IBD matrices:\\
	 B1 with  MD$+1$ and  2nd SB$-1$, 
         B2 with MD$+1$ and 1st SP$-1$, and
	B3 with  MD$+1$ and  8th SP$+1$
\item Class 8:	Sum of a 3-ASPH  and three IBD matrices:\\
	B1 with   MD$+1$ and 1st SB$-1$,
	B2 with   MD$+1$ and 1st SP$-1$, and
	B3 with   MD$+1$ and 4th  SB$+1$
\item Class 9:	Sum of a 3-ASPH  and three IBD matrices:\\
	B1 with   MD$+1$ and 1st SB$+1$,
	B2 with   MD$+1$ and 1st SP$-1$, and
	B3 with   MD$-1$ and 3rd SP$+1$
\item Class 10:	Sum of three IBD matrices:\\
	B1 with   MD$+1$ and 1st SB$+1$,
	 B2 with   MD$+1$ and 1st SP$-1$, and
	 B3 with   MD$-1$ and 3rd SP$+1$
\item Class 11:	Sum of a 3-APH  and three IBD matrices:\\
	 B1 with   MD$+1$ and  2nd SB$-1$,
	 B2 with   MD$+1$ and 1st SP$-1$, and
	 B3 with   MD$+1$ and  8th SP$+1$
\item Class 12:	Sum of a 3-APH  and two IBD matrices:\\
	 B1 with   MD$+1$ and 1st SB$-1$ and
	 B2 with   MD$+1$ and 1st SP$-1$
\item Class 13:	Sum of a 3-ASPH  and a permutation matrix
\item Class 14:	Sum of a 3-ASPH  and two permutation matrices
\item Class 15:	Sum of a 3-ASPH  and three permutation matrices
\item Class 16:	Sum of a 3-APH  and three  permutation matrices
\item Class 17:	Sum of a 3-APH  and two permutation matrices
\end{itemize}

\begin{figure}[htb] 
\centering
\caption{Relative Error Norm For SVD Matrices}
\label{SuperfastSVD}
\end{figure}

\begin{figure}[htb] 
\centering
\caption{Relative Error Norm For Lapacian Matrices}
\label{SuperfastLP}
\end{figure}

\begin{figure}[htb] 
\centering
\caption{Relative Error Norm For Finite-Difference Matrices}
\label{SuperfastFD}
\end{figure}

\begin{table}[htb] \label{SuperfastTable}
\begin{center}
\begin{tabular}{|c|c|c|c|c|c|c|}
\hline
			& \multicolumn{2}{|c|}{SVD Matrices} & \multicolumn{2}{|c|}{Laplacian Matrices} & \multicolumn{2}{|c|}{Finite Difference Matrices}\\\hline
 \text{Class No.} & \text{Mean} & \text{Std} & \text{Mean} & \text{Std} & \text{Mean} & \text{Std} \\\hline
Class 0	&	3.54E-09	&	3.28E-09	&	4.10E-14	&	2.43E-13	&	1.61E-06	&	1.35E-06\\\hline
Class 0	&	1.07E-08	&	3.82E-09	&	2.05E-13	&	1.62E-13	&	4.58E-06	&	9.93E-07\\\hline
Class 1	&	1.16E-08	&	6.62E-09	&	6.07E-13	&	5.20E-13	&	4.67E-06	&	1.04E-06\\\hline
Class 2	&	1.23E-08	&	5.84E-09	&	1.69E-13	&	1.34E-13	&	4.52E-06	&	1.01E-06\\\hline
Class 3	&	1.25E-08	&	1.07E-08	&	2.46E-13	&	3.44E-13	&	4.72E-06	&	9.52E-07\\\hline
Class 4	&	1.13E-08	&	6.09E-09	&	1.93E-13	&	1.48E-13	&	4.38E-06	&	8.64E-07\\\hline
Class 5	&	1.12E-08	&	8.79E-09	&	9.25E-13	&	2.64E-12	&	5.12E-06	&	1.29E-06\\\hline
Class 6	&	1.16E-08	&	7.42E-09	&	5.51E-13	&	5.35E-13	&	4.79E-06	&	1.12E-06\\\hline
Class 7	&	1.33E-08	&	1.00E-08	&	1.98E-13	&	1.30E-13	&	4.60E-06	&	9.52E-07\\\hline
Class 8	&	1.08E-08	&	4.81E-09	&	2.09E-13	&	3.60E-13	&	4.47E-06	&	8.57E-07\\\hline
Class 9	&	1.18E-08	&	5.51E-09	&	1.87E-13	&	1.77E-13	&	4.63E-06	&	9.28E-07\\\hline
Class 10	&	1.18E-08	&	6.23E-09	&	1.78E-13	&	1.42E-13	&	4.55E-06	&	9.08E-07\\\hline
Class 11	&	1.28E-08	&	1.40E-08	&	2.33E-13	&	3.44E-13	&	4.49E-06	&	9.67E-07\\\hline
Class 12	&	1.43E-08	&	1.87E-08	&	1.78E-13	&	1.61E-13	&	4.74E-06	&	1.19E-06\\\hline
Class 13	&	1.22E-08	&	1.26E-08	&	2.21E-13	&	2.83E-13	&	4.75E-06	&	1.14E-06\\\hline
Class 14	&	1.51E-08	&	1.18E-08	&	3.57E-13	&	9.27E-13	&	4.61E-06	&	1.08E-06\\\hline
Class 15	&	1.19E-08	&	6.93E-09	&	2.24E-13	&	1.76E-13	&	4.74E-06	&	1.09E-06\\\hline
Class 16	&	1.26E-08	&	1.16E-08	&	2.15E-13	&	1.70E-13	&	4.59E-06	&	1.12E-06\\\hline
Class 17	&	1.31E-08	&	1.18E-08	&	1.25E-14	&	5.16E-14	&	1.83E-06	&	1.55E-06\\\hline

\end{tabular}
\caption{Relative Error Norm with Superfast Multipliers}
\end{center}
\end{table}

The tests show quite accurate outputs even where we applied Algorithm \ref{alg1}
with very sparse multipliers of classes 13--17.

%------------------------------------------------------------------------------

\clearpage

%------------------------------------------------------------------------------
 
\section{Conclusions: Three Extensions}\label{sext}

%------------------------------------------------------------------------------

Our duality techniques for  the average inputs 
can  be extended to the acceleration of various 
matrix computations.
In this concluding section we describe   
 extensions to  three highly important areas.
%In the next subsection we describe such an extension to the  
%LSR problem. 
%In Section \ref{sfmm} we apply our dual techniques
%in order to accelerate the bottleneck stage of 
%a basic application of the FMM. 
%In Section \ref{scg} we accelerate  
%the CG algorithms by reducing them to the FMM.

%------------------------------------------------------------------------------
 
\subsection{Acceleration of the computation for Least Squares Regression (LSR)}\label{slsr}

%------------------------------------------------------------------------------

We first recall the following fundamental problem of matrix computations  (cf. \cite{GL13}):

\begin{problem}\label{pr1} {\em Least Squares Solution of an Overdetermined Linear System of Equations.}
Given two integers $m$ and $d$ such that $1\le d<m$,
a matrix $A\in \mathbb R^{m\times d}$, and a vector ${\bf b}\in \mathbb R^{m}$,
compute a vector ${\bf x}$ that minimizes the norm $||A{\bf x}-{\bf b}||$.
\end{problem}

If a matrix $A$ has full rank $n$, then unique solution 
is given by the vector ${\bf x}=(A^TA)^{-1}A^T{\bf b}$,
satisfying the linear system of normal equations $A^TA{\bf x}=A^T{\bf b}$.
Otherwise solution is not unique, and a solution  ${\bf x}$ having the
minimum norm is given by the vector $A^+{\bf b}$. 
%------------------------------------------------------------------------------
In the important case where $m\gg d$ and an approximate solution
is acceptable,  Sarl\'os in \cite{S06} proposed  to
  simplify the computations as follows:

%------------------------------------------------------------------------------
 
\begin{algorithm}\label{algapprls} 
{\rm  Least Squares Regression (LSR).}

%------------------------------------------------------------------------------
 
\begin{description}

%------------------------------------------------------------------------------

%\item[{\sc Input:}] 
%As for Problem \ref{pr1}.
%------------------------------------------------------------------------------

%\item[{\sc Output:}] 
%A  rank-$r$  approximation matrix $\tilde M$ 
%and the relative error $||\tilde M-M||/||M||$. 

%------------------------------------------------------------------------------

\item[{\sc Initialization:}] 
 Fix an 
%oversampling
 integer $k$ such that $1\le k\ll m$. 

%------------------------------------------------------------------------------ 

\item[{\sc Computations:}]
%\item (cf. items 1.4.6 and 1.4.10): $~$

\begin{enumerate}
\item %1
Generate a scaled $k\times m$ Gaussian matrix $F$.

\item %2 
Compute the matrix $FA$ and the vector $F{\bf b}$.
\item %3
Output
a solution $\tilde {\bf x}$ to the compressed Problem \ref{pr1}
where the matrix $A$ and the vector ${\bf b}$
are replaced by the matrix 
$FA$ and the vector $F{\bf b}$,
respectively.
\end{enumerate}

%------------------------------------------------------------------------------

\end{description}

%------------------------------------------------------------------------------

\end{algorithm}

%------------------------------------------------------------------------------

Now write $M=(A~|~{\bf b})$ and ${\bf y}=\begin{pmatrix}
{\bf x} \\ -1
\end{pmatrix}$
and compare the error norms 
$||FM\tilde{\bf y}||=||FA\tilde{\bf x}-F{\bf b}||$ (of the output $\tilde{\bf x}$ 
of the latter algorithm) and $||M{\bf y}||=||A{\bf x}-{\bf b}||$
(of the solution ${\bf x}$ of the original Problem \ref{pr1}).

%------------------------------------------------------------------------------
 
\begin{theorem}\label{thlsrd}  {\em  \cite[Theorem 2.3]{W14}.}
Suppose that we are given 
two tolerance values $\delta$ and  $\xi$, $0<\delta<1$ and  $0<\xi<1$, 
three integers $k$, $m$ and  $d$ such that $1\le d<m$
and $$k=(d+\log(1/\delta)\xi^{-2})\theta,$$ for a certain constant $\theta$, and
a matrix $G_{k,m}\in \mathcal G^{k\times m}$. 
Then, with a probability at least $1-\delta$, it holds that 
%\begin{equation}\label{elsrd}
$$(1-\xi)||M{\bf y}||\le \frac{1}{\sqrt k}||G_{k,m}M{\bf y}||\le (1+\xi)||M{\bf y}||$$
%\end{equation}
for all matrices $M\in \mathbb R^{m\times (d+1)}$
and all vectors ${\bf y}=(y_i)_{i=0}^{d}\in \mathbb R^{d+1}$ normalized so that $y_d=-1$.
\end{theorem}

 The
theorem implies that with a probability at least $1-\delta$,
Algorithm \ref{algapprls} outputs an approximate solution to Problem \ref{pr1}
within the error norm bound $\xi$
provided that $k=(d+\log(1/\delta)\xi^{-2})\theta$
and $F=\frac{1}{\sqrt k}G_{k,m}$.\footnote{Such approximate solutions 
serve as preprocessors for practical implementation of
numerical linear algebra algorithms 
for Problem \ref{pr1} of least squares computation \cite[Section 4.5]{M11}, \cite{RT08}, \cite{AMT10}.}

 For $m\gg k$, the computational cost of performing the algorithm for approximate 
solution  
dramatically decreases versus the cost of computing exact solution,
but can still be prohibitively high at the stage of computing the matrix product
$FM$. 
In a number of papers the former cost has been dramatically 
decreased further by means of replacing a multiplier 
$F=\frac{1}{\sqrt k}G_{k,m}$ with various random sparse and structured matrices
(see \cite[Section 2.1]{W14}), for which the bound of Theorem \ref{thlsrd}
%(\ref{elsrd})
still holds for all matrices
$M\in \mathbb R^{m\times (d+1)}$, although
at the expense of increasing significantly the dimension $k$.  

Can we achieve similar progress without such an increase?
 The following theorem provides positive answer
 in the case where $M$ is  the average matrix in $\mathbb R^{m\times (d+1)}$
under the Gaussian probability distribution:

%------------------------------------------------------------------------------
 
\begin{theorem}\label{thlsrdd} {\em Dual LSR.}
The bound 
%(\ref{elsrd}) 
of Theorem \ref{thlsrd} holds with a probability 
at least $1-\delta$ where $\sqrt k~ M\in \mathcal G^{m\times (d+1)}$
and $F\in \mathbb R^{k\times m}$ is an orthogonal matrix.
\end{theorem}
\begin{proof} 
Theorem \ref{thlsrd} has been proven in \cite[Section 2]{W14} in
the case where  
$\sqrt k ~FM\in \mathcal G^{k\times (d+1)}$.
 This is the case where $\sqrt k ~F\in  \mathcal G^{k\times m}$
and $M$ is an orthogonal $m\times (d+1)$ matrix,
but  is also the case under the assumptions of  
Theorem \ref{thlsrdd}, by virtue of Lemma \ref{lepr3}. 
\end{proof}

Theorem \ref{thlsrdd} supports the computation of
LSR
for any orthogonal multiplier $F$
(e.g., a Hadamard's scaled multiplier or
a CountSketch multiplier from
the data stream literature \cite[Section 2.1]{W14},
\cite{CCF04}, \cite{TZ12})
and for an input matrix $M\in \mathbb R^{m\times (d+1)}$,
 average under the Gaussian probability distribution.

It follows that
Algorithm \ref{algapprls} can fail only  
for a narrow class of pairs $F$ and $M$ where
$F$ denotes orthogonal matrices in $\mathbb R^{k\times m}$
 and $M$ denotes matrices in $\mathbb R^{m\times (d+1)}$,
and even in the case of failure we can still have good chances to succeed by using 
heuristic recipes of our Section \ref{smngflr}.

\subsection{Acceleration of the Fast Multipole Method (FMM) for the Average HSS matrix}\label{sfmm}

%------------------------------------------------------------------------------

Next  
 we point out how our duality approach can accelerate  
a fundamental application of the FMM  to 
multiplication by a vector of  HSS matrix \footnote{Here and hereafter ``{\em HSS}" stands for ``hierarchically semiseparable".} defined by its average case generators. This preliminary sketch
(together with its further extension to 
the CG algorithms in the next subsection) should 
demonstrate new chances for the expansion of 
the large application area of low-rank approximation.

We first recall that
 HSS matrices
 naturally extend the class of banded matrices and their inverses, 
are closely linked to FMM,  
have been intensively studied for decades
(cf.   
 \cite{CGR88}, \cite{GR87}, 
 \cite{T00}, 
 \cite{BGH03},  \cite{GH03},
 \cite{VVGM05}, 
\cite{VVM07/08}, 
 \cite{B10}, \cite{X12},   \cite{XXG12},
% \cite{BY13},
\cite{EGH13}, \cite{X13},
 \cite{XXCB14},
  and the bibliography therein), and 
are highly and increasingly popular. 

\begin{definition}\label{defneut} (Cf. 
\cite{MRT05}.) With each diagonal block of a block matrix 
associate 
its complement in its block column,
that is,
the union of the pair of the maximal sub- and super-diagonal blocks
in that column block, and call this complement a {\em neutered block column}.
\end{definition}

\begin{definition}\label{defqs} (Cf. 
% \cite{GR87}, \cite{CGR88}, \cite{T00}, 
% \cite{BGH03},  \cite{GH03}, \cite{B10},
 \cite{X12},  \cite{X13}, \cite{XXCB14}.)

A block 
matrix $M$ of size  $m\times n$ is 
called a $r$-{\em HSS matrix}, for a positive integer $r$, 

(i) if all diagonal blocks of this matrix
consist of $O((m+n)r)$ entries overall
and
 
(ii) if $r$ is the maximum rank of its neutered blocks. 
\end{definition}

\begin{remark}\label{reqs}
Many authors work with $(l,u)$-HSS
rather than $r$-HSS matrices $M$ where $l$ and $u$
are the maximum ranks of the sub- and super-block-diagonal blocks,
respectively.
The $(l,u)$-HSS and $r$-HSS matrices are closely related. 
If a neutered block column $N$
is the union of a sub-block-diagonal block $B_-$ and 
a super-block-diagonal block $B_+$,
then
 $\rank (N)\le \rank (B_-)+\rank (B_+)$,
 and so
an $(l,u)$-HSS matrix is a $r$-HSS matrix,
for $r\le l+u$,
while clearly a $r$-HSS matrix is  
a $(r,r)$-HSS matrix.
\end{remark}

The FMM exploits the $r$-HSS structure of a matrix as s
(cf.  \cite{VVM07/08}, \cite{B10}, \cite{EGH13}):

(i) Cover all off-block-diagonal entries
with a set of  non-overlapping neutered block columns.  

(ii) Express every neutered block column $N$ of this set
  as the product
  $FH$ of two  {\em generator
matrices}, $F$ of size $h\times r$
and $H$ of size $r\times k$. Call the 
pair $\{F,H\}$ a {\em length $r$ generator} of the 
neutered block column $N$. 

(iii)  Multiply readily
the matrix $M$ by a vector by separately multiplying generators
and diagonal blocks by subvectors, which involves $O((m+n)r)$ flops
overall, and

(iv) in a more advanced application of  
 FMM solve a nonsingular $r$-HSS linear system of $n$
equations  by using
$O(nr\log^2(n))$ flops under some mild additional assumptions on  the input.

This approach is readily extended to the same operations with 
$(r,\xi)$-{\em HSS matrices},
that is, matrices approximated by $r$-HSS matrices
within a perturbation norm bound $\xi$ where a  positive tolerance 
$\xi$ is small in context (e.g., is the unit round-off).
 Likewise, one defines
an  $(r,\xi)$-{\em HSS representation} and 
$(r,\xi)$-{\em generators}.

$(r,\xi)$-HSS matrices (for $r$ small in context)
appear routinely in matrix computations,
and computations with such matrices are 
performed  efficiently by using the
above techniques.

%------------------------------------------------------------------------------

In many applications of the FMM (cf., e.g., \cite{BGP05}, \cite{VVVF10})
stage (ii) is omitted because short generators for all 
neutered block columns are readily available,
 but this is not the case in other important applications
 (cf. \cite{XXG12}, \cite{XXCB14}, \cite{P15}, and our Section \ref{scg}).
 The computation of such generators is precisely the
 low-rank approximation of the neutered block
columns, which turns out to be 
the 
bottleneck stage of FMM in these applications. 

Indeed 
 apply random sampling Algorithm \ref{alg1} at this
 stage with Gaussian multipliers.
Multiplication of a $q\times h$ matrix
by $h\times r$ Gaussian matrix requires $(2h-1)qr$ flops,
while
standard HSS-representation of $n\times n$ 
HSS matrix includes $q\times h$ neutered 
 block columns for $q\approx m/2$ and $h\approx n/2$. In this case 
the cost of computing an $r$-HSS representation of the
matrix $M$ is at least of order $mnr$.
For large integers $m$ and $n$, this
 greatly exceeds 
the above estimate of $O((m+n)r)$ flops at the other stages of 
the computations. 

Alternative customary techniques for low-rank approximation
rely on computing SVD
or rank-revealing factorization of an input matrix 
and are at least as costly as the computations by means of random sampling.

Can we alleviate such a problem? Yes, 
we  can  accelerate 
 randomized computation
of the generators  
for  the average neutered block columns involved into a  desired  $(r,\xi)$-HSS
representation,
by applying 
 our recipes of Section \ref{smngflr}
and  multipliers of Section \ref{ssprsml}.
The papers \cite{XXG12} and \cite{XXCB14} have succeeded
by means of ad hoc application of random Toeplitz multipliers.
Now our Dual Theorem \ref{th1d} backs up that success,
and our recipes of Section \ref{smngflr}
enable further simplification of such computations.

%------------------------------------------------------------------------------
%------------------------------------------------------------------------------
 
\subsection{Acceleration of the Conjugate Gradient (CG) algorithms}\label{scg} 

%------------------------------------------------------------------------------

%Next we extend our progress for  the FMM to the CG Algorithms.
% by means of their simple but apparently novel reduction to.

We recall that a  real $n\times n$ matrix $M$ and a  
linear system of $n$ equations $M{\bf x}={\bf b}$
are said to be  {\em symmetric positive definite}\footnote{Hereafter 
we use the acronym {\em spd}.} if
$M=V^TV$ for a nonsingular matrix $V$
\cite{GL13} 
and use the following concept:
 
%------------------------------------------------------------------------------

\begin{definition}\label{defconc}
An $n\times n$ matrix $M$ is $(r,\xi)$-{\em concentrated} if 
the set 
of its
 singular values is clustered (within a small tolerance $\xi$)
about at most  $r+1$ values. 
Such a matrix is {\em strongly} $(r,\xi)$-{\em concentrated} if this
 set contains a cluster of at least $n-r$  singular values,
each counted with its multiplicity.
\end{definition}

The following two facts imply the efficiency of the CG  method:  

\medskip

(i) Given a spd linear system
of  equations $M{\bf x}={\bf b}$
whose matrix $M$ 
is $(r,\xi)$-concentrated, the CG algorithms converge  to its solution
within a error norm in $O(\xi)$ in at most $r$ 
iterations  \cite{A94}, 
% \cite{BBC93},  
\cite{G97},
% \cite{B02}, 
\cite{S03}.  

\medskip

(ii) Various highly important 
 present day computations routinely  involve matrices 
made strongly $(r,\xi)$-concentrated and hence 
 $(r,\xi)$-concentrated,
 for reasonably small integers $r$
and small positive $\xi$,
by means of some standard preconditioning techniques.

\medskip

The next critical issue is whether we can decrease 
the computational cost of a CG iteration, which is reduced 
essentially to computing or closely approximating
 the product
of the matrix $M$ by a vector.

Next we 
prove that strongly $(r,\xi)$-concentrated matrices
are also $(r,\xi)$-HSS matrices, and so, by applying our
 accelerated variant of FMM,  we can approximate the product of
 such an $n\times n$ average matrix $M$ by a vector
 significantly faster than
%using $O(rn\log (n))$ flops, versus 
 by applying the known algorithms,
which involve $(2n-1)n$ flops.

%------------------------------------------------------------------------------

\begin{theorem}\label{thnroff}
If an $n\times n$ spd matrix $M$ is strongly 
 $r$-concentrated, then the numerical rank of any of its off-diagonal submatrix is at most $r$.
\end{theorem}

\begin{proof}
Since $M$ is a strongly 
 $r$-concentrated matrix, at least $n-r$ its singular values
 are clustered about a certain value $s$.
Change all the singular
values by assigning to them
this value  $s$  
and denote the resulting matrix $\widehat M=M+E$
where the matrix $E$ has numerical rank at most $r$ 
and where $\widehat M=U\widehat SU^T$, for
$\widehat S =s I_n$ and an orthogonal matrix $U$.
Note that in this case $\widehat M=sI_n$, and so every
off-diagonal submatrix of $\widehat M$ is filled with 0s.
Therefore the matrices $M$ and $E$ share
all their off-diagonal submatrices.
Consequently numerical rank of such a submatrix cannot exceed 
  $\nrank(E)\le r$.
\end{proof}

\begin{corollary}\label{cofmm}
If an $n\times n$ spd matrix $M$ is strongly 
 $r$-concentrated, then 
 we can  approximate the solution 
of a linear system of $n$ equations $M{\bf x}={\bf b}$
by using 
$O(r^2n\log(n))+\gamma(M)$ flops
provided that one can compute generators 
of length at most $r$ for a $r$-HSS approximate representation
of the matrix  $M$ by using $\gamma(M)$ flops.
\end{corollary}

By virtue of Theorem \ref{thnroff}  the neutered block columns 
in the $r$-HSS representation of the matrix $M$
have numerical ranks at most $r$.
By virtue of our results of the 
previous subsection,
$\gamma(M)=O(rn)$ in the case of the average matrix $M$, 
 implying respective acceleration of the  CG algorithms.

\begin{remark}\label{refmm1} {\rm Extension to nonsymmetric inputs.}
Recall that any nonsingular linear system
$V{\bf x}={\bf b}$ is equivalent to the spd linear systems
$V^TV{\bf x}={\bf c}$ and $VV^T{\bf y}={\bf b}$ for
${\bf c}=V^T{\bf b}$ and ${\bf x}=V^T{\bf y}$.
Therefore we can extend our results to a 
nonsymmetric nonsingular linear system
of equations $A{\bf y}={\bf f}$ by means of
symmetrization of a matrix $A$ in any of the two ways,
$A\rightarrow M=A^TA$ or $A\rightarrow M=AA^T$, and then application of 
the CG algorithms to the matrix $M$
defined implicitly by the above products,
and never computed explicitly.
This leads to
the {\em CG normal equation error} method and the
{\em CG normal equation error} method, respectively
 (cf. \cite[Section 11.3.9]{GL13}),
to which we can extend our study of the CG method.
\end{remark}

\medskip
%------------------------------------------------------------------------------

\medskip

%------------------------------------------------------------------------------

{\bf {\LARGE {Appendix}: Definitions and Auxiliary Results}}
\appendix 

\section{General Matrices}\label{sdef}

%------------------------------------------------------------------------------

For simplicity, in the Appendix  we mostly assume dealing with real matrices,
but  can  readily extend all our study to the complex case.

\begin{enumerate}
\item%1
$I_g$ is a $g\times g$ identity matrix.
$O_{k,l}$ is the  $k\times l$ matrix filled with zeros.
\item%2 
  $(B_1~|~B_2~|~\cdots~|~B_k)$ is a block vector of length $k$, 
and
$\diag(B_1,B_2,\dots,B_k)$ is a $k\times k$ block
diagonal matrix, in both cases with blocks $B_1,B_2,\dots,B_k$.

% with blocks $B_1,\dots,B_k$.
\item%3
 $W_{k,l}$ denotes the
$k\times l$ leading (that is, northwestern) block
of an $m\times n$ matrix $W$.
% for $k\le m$ and $l\le n$. 
\item%4  
%(In this case the matrix
%$W^T$ is orthogonal as well.)
%\item%6
 %$Q(W)$ denotes the matrix 
%obtained by means of column orthogonalization
%of a matrix $W$. 
%including the deletion of 
%the columns filled with zeros.
$W=S_{W,\rho}\Sigma_{W,\rho}T^T_{W,\rho}$ is 
 {\em  compact SVD}
of a matrix $W$ of rank
$\rho$ with
 $S_{W,\rho}$ and $T_{W,\rho}$   real orthogonal or  unitary
matrices of its singular vectors and
$\Sigma_{W,\rho}=\diag(\sigma_j(W))_{j=1}^{\rho}$ the
diagonal matrix of its singular values
in non-increasing order,
$\sigma_1(W)\ge \sigma_2(W)\ge \dots\ge 
\sigma_{\rho}(W)>0$.
\item%9
 $W^+=T_{W,\rho}\Sigma_{W,\rho}^{-1}S^T_{W,\rho}$ 
is the  {\em Moore--Penrose pseudo
inverse} of the matrix $W$. 
%$(W^+)^+=W$, $WW^+=I_m$ if $\rank (W)=m$,
%$W^+W=I_n$ if $\rank (W)=n$,
%($W^+=W^{-1}$ for a nonsingular matrix $W$.)
\item%10
$||W||=\sigma_1(W)$ and
$||W||_F=(\sum_{j=1}^{\rho}\sigma_j^2(W))^{1/2} 
%=Tr$(W^HW))^{1/2}$
\le \sqrt n~||W||$ 
denote its  {\em spectral and Frobenius norms}, respectively.
%Recall that 
($||W^+||=\frac{1}{\sigma_{\rho}(W)}$; 
%and that 
%$||U||=||U^+||=1$,
$||UW||=||W||$ and $||WU||=||W||$ if the matrix $U$ is unitary.)
%$||VW||\le ||V||~||W||$ and $||VW||_F\le ||V||_F||W||_F$,
%for any matrix product $VW$.
\item%11
 $\kappa(W)=||W||~||W^+||=\sigma_1(W)/\sigma_{\rho}(W)\ge 1$
 denotes the  {\em condition number}
of a matrix $W$. 
%\item%12
%The  $\xi$-rank of a matrix, for a fixed positive  $\xi$,
%is the minimum rank of its approximations within 
%the norm bound  $\xi$. 
%The   {\em numerical rank} of a matrix is its  $\xi$-rank
%for $\xi$ being small in context.
%\item%13
A matrix
 is called
{\em ill-conditioned} if 
its condition number 
 is large in context and
%or equivalently if its rank exceeds its numerical rank. 
is called {\em well-conditioned}
if  this number
$\kappa(W)$ is reasonably bounded.
(An $m\times n$ matrix is ill-conditioned
if and only if it has a matrix of a smaller rank nearby, 
and it is well-conditioned 
if and only if it has full numerical rank $\min\{m,n\}$.)
%(The ratio of the output and input error norms of Gaussian elimination 
%is roughly the condition number of an input matrix,
%cf.  \cite{GL13}.)
\item%9
%------------------------------------------------------------------------------
The following theorem is implied by \cite[Corollary 1.4.19]{S98} for $P= -C^{-1}E$:
\begin{theorem}\label{thpert} 
Suppose $C$ and $C+E$ are two nonsingular matrices of the same size
and $$||C^{-1}E||=\theta<1.$$ Then
%\begin{itemize}
%\item %1
%$||I-(C+E)^{-1}C||  \le \frac{\theta}{1-\theta}$ and
$$\||(C+E)^{-1}-C^{-1}||\le \frac{\theta}{1-\theta}||C^{-1}||;$$
%\item %2
e.g., $\||(C+E)^{-1}-C^{-1}||\le 0.5||C^{-1}||$
if $\theta\le 1/3$.
%\end{itemize}
\end{theorem}
%------------------------------------------------------------------------------
\end{enumerate}
%\end{itemize}
%Hereafter we refer to items 1--14 above as items 2.1--2.14.

%------------------------------------------------------------------------------

\section{Gaussian Matrices}\label{srndm}

%------------------------------------------------------------------------------

\begin{theorem}\label{thrnd} 
Suppose that $A$ is an  
$m\times n$ matrix of full rank $k=\min\{m,n\}$, 
$F$ and $H$ are $r\times m$ and 
$n\times r$  matrices, respectively, for $r\le k$,
and  the  entries of these two matrices are nonconstant linear combinations
of finitely many i.i.d. random variables $v_1,\dots,v_h$. 

Then 
the matrices $F$, $FA$, $H$, and $AH$  
have full rank $r$ 

(i) with probability 1
if $v_1,\dots,v_h$ are Gaussian variables
and 

(ii) with a probability at least $1-r/|\mathcal S|$
if they are random variables sampled under the
uniform probability distribution from 
a finite set $\mathcal S$ having cardinality 
$|\mathcal S|$.  
\end{theorem}

%------------------------------------------------------------------------------

\begin{proof}
The determinant, $\det(B)$,
of any $r\times r$ block $B$ of a matrix 
 $F$, $FA$, $H$, or $AH$
is a polynomial of degree $r$ in the variables  $v_1,\dots,v_h$,
and so the equation $\det(B)=0$ 
  defines an algebraic variety of a lower
dimension in the linear space of these variables
(cf. \cite[Proposition 1]{BV88}). 
Clearly, such a variety has Lebesgue  and
Gaussian measures 0, both being absolutely continuous
with respect to one another. This implies part (i) of the theorem.
Derivation of part (ii) from a celebrated lemma of \cite{DL78},
 also known from \cite{Z79} and \cite{S80},  
is a well-known pattern, specified in some detail
 in \cite{PW08}.
\end{proof}

%------------------------------------------------------------------------------

\begin{lemma}\label{lepr3} ({\rm Rotational invariance of a Gaussian matrix.})
%\cite[Proposition 2.2]{SST06}.
Suppose that $k$, $m$, and $n$  are three  positive integers,
$G$ is an 
 $m\times n$  Gaussian matrix, and
$S$ and $T$ are $k\times m$ and 
$n\times k$
 orthogonal matrices, respectively.

Then $SG$ and $GT$ are Gaussian matrices.
\end{lemma}

%------------------------------------------------------------------------------

We state the following results and estimates for real matrices,
but similar estimates in the case of complex  
 matrices can be found in \cite{D88}, \cite{E88}, \cite{CD05}, 
and \cite{ES05}:

\begin{definition}\label{defnrm} {\rm Norms of random matrices and
expected value of 
a random variable.}
Write 
$\nu_{m,n}=||G||$,
$\nu_{m,n}^+=||G^+||$,  and
$\nu_{m,n,F}^+=||G^+||_F$,
for  a  Gaussian $m\times n$ matrix  $G$,
and write $\mathbb E(v)$ for the expected value of 
a random variable $v$.
($\nu_{m,n}=\nu_{n,m}$,
$\nu_{m,n}^+=\nu_{n,m}^+$,
and $\nu_{F,m,n}=\nu_{F,n,m}$,
for all pairs of $m$ and $n$.) 
\end{definition}

%-----------------------------------------------------------------------------

\begin{theorem}\label{thsignorm}
(Cf. \cite[Theorem II.7]{DS01}.)
Suppose 
that $m$ and $n$ are positive integers,
$h=\max\{m,n\}$, $t\ge 0$.
%------------------------------------------------------------------------------ 
Then 

(i) {\rm Probability}$\{\nu_{m,n}>t+\sqrt m+\sqrt n\}\le
\exp(-t^2/2)$ and 

(ii) $\mathbb E(\nu_{m,n})< 1+\sqrt m+\sqrt n$.
\end{theorem}

%------------------------------------------------------------------------------

\begin{theorem}\label{thsiguna} 
Let $\Gamma(x)=
\int_0^{\infty}\exp(-t)t^{x-1}dt$
denote the Gamma function and let  $x>0$. 
%{\rm and}~\zeta(t)=
%\frac{\sqrt{2m}}{\Gamma(m/2)}(t\sqrt{m/2})^{m-1}\exp(-mt^2/2)=
%2~t^{m-1}(\frac{m}{2})^{m/2}\exp(-\frac{m}{2}t^2)/\Gamma(\frac{m}{2}).$$ 
Then 

(i)  {\rm Probability} $\{\nu_{m,n}^+\ge m/x^2\}<\frac{x^{m-n+1}}{\Gamma(m-n+2)}$
for $m\ge n\ge 2$,

(ii) {\rm Probability} $\{\nu_{n,n}^+\ge x\}\le 2.35 {\sqrt n}/x$ 
for $n\ge 2$,

(iii)  $\mathbb E(\nu^+_{m,n})\le e\sqrt{m}/|m-n|$,
provided that $m\neq n$ and $e=2.71828\dots$.
%Probability $\{||(G+A)^+||\ge 2.35x\sqrt {n}\}\le 1/x$
%for any $m\times n$ matrix $A$.
\end{theorem}

%------------------------------------------------------------------------------

\begin{proof}
 See \cite[Proof of Lemma 4.1]{CD05} for part (i), 
\cite[Theorem 3.3]{SST06} for part (ii),  and
\cite[Proposition 10.2]{HMT11} for part (iii).
\end{proof}

%------------------------------------------------------------------------------
 
The  probabilistic upper bounds of Theorem \ref{thsiguna}
on $\nu^+_{m,n}$ are reasonable already  
in the case of square matrices, that is, where $m=n$,
but become much stronger as the difference $|m-n|$ grows large.

%------------------------------------------------------------------------------

Theorems \ref{thsignorm} and \ref{thsiguna}
combined imply that an $m\times n$ Gaussian
matrix is well-conditioned  
unless the integer $m+n$ is large or the integer $|m-n|$ is close to 0.
With some grain of salt
we can still consider  such a matrix  
well-conditioned  
 even
  where the integer $|m-n|$ is small or  vanishes
provided that  the integer $m$ is not large.
Clearly, these properties can be extended immediately to all submatrices.

% - - - - - - - - - - - - - - - - - - - - - - - - - - - - - - - - - - - - - - -
%------------------------------------------------------------------------------

\medskip

%------------------------------------------------------------------------------

\noindent {\bf Acknowledgements:}
Our research has been supported by NSF Grant CCF 1116736 
and PSC CUNY Award  68862--00 46.
%We are also grateful to the reviewers for valuable comments.

%------------------------------------------------------------------------------

%------------------------------------------------------------------------------

\end{document}